\documentclass[a4paper,11pt]{amsart}

\usepackage{amsmath, amsthm, amsfonts, amssymb}
\usepackage[bookmarks=false]{hyperref}
\usepackage{enumerate}

\usepackage{color}

\numberwithin{equation}{section}

\newtheorem{letterthm}{Theorem}

\newtheorem{lettercor}[letterthm]{Corollary}

\newtheorem{letterconj}[letterthm]{Conjecture}

\newtheorem{thm}{Theorem}[section]
\newtheorem{lem}[thm]{Lemma}

\newtheorem{cor}[thm]{Corollary}
\newtheorem{prop}[thm]{Proposition}

\theoremstyle{definition}
\newtheorem{rem}[thm]{Remark}

\newtheorem{df}[thm]{Definition}

\newtheorem*{definition}{Definition}
\newtheorem{claim}[thm]{Claim}

\newcommand{\R}{\mathbf{R}}
\newcommand{\C}{\mathbf{C}}

\newcommand{\F}{\mathbf{F}}
\newcommand{\Q}{\mathbf{Q}}
\newcommand{\N}{\mathbf{N}}

\newcommand{\cM}{\mathcal{M}}

\newcommand{\Ad}{\operatorname{Ad}}
\newcommand{\id}{\text{\rm id}}

\newcommand{\Inn}{\operatorname{Inn}}
\newcommand{\Aut}{\operatorname{Aut}}
\newcommand{\Out}{\operatorname{Out}}
\newcommand{\rL}{\mathord{\text{\rm L}}}
\newcommand{\rC}{\mathord{\text{\rm C}}}

\newcommand{\diag}{\mathord{\text{\rm diag}}}

\newcommand{\rE}{\mathord{\text{\rm E}}}

\newcommand{\Sd}{\mathord{\text{\rm Sd}}}

\newcommand{\alg}{\text{alg}}

\newcommand{\Tr}{\mathord{\text{\rm Tr}}}

\newcommand{\Ball}{\mathord{\text{\rm Ball}}}
\newcommand{\spn}{\mathord{\text{\rm span}}}

\newcommand{\ovt}{\mathbin{\overline{\otimes}}}

\newcommand{\bigovt}{\mathbin{\overline{\bigotimes}}}

\newcommand{\vphi}{\varphi}

\newcommand{\II}{{\rm II}}
\newcommand{\III}{{\rm III}}

\begin{document}

\title[Fullness and Connes' $\tau$ invariant of type III tensor product factors]{Fullness and Connes' $\boldsymbol{\tau}$ invariant of \\ type III tensor product factors}

\begin{abstract}
We show that the tensor product $M \ovt N$ of any two full factors $M$ and $N$ (possibly of type ${\rm III}$) is full and we compute Connes' invariant $\tau(M \ovt N)$ in terms of $\tau(M)$ and $\tau(N)$. The key novelty is an enhanced spectral gap property for full factors of type $\III$. Moreover, for full factors of type $\III$ with almost periodic states, we prove an optimal spectral gap property. As an application of our main result, we also show that for any full factor $M$ and any non-type ${\rm I}$ amenable factor $P$, the tensor product factor $M \ovt P$ has a unique McDuff decomposition, up to stable unitary conjugacy. 
\end{abstract}

\address{Laboratoire de Math\'ematiques d'Orsay\\ Universit\'e Paris-Sud\\ CNRS\\ Universit\'e Paris-Saclay\\ 91405 Orsay\\ FRANCE}

\author{Cyril Houdayer}
\email{cyril.houdayer@math.u-psud.fr}

\author{Amine Marrakchi}
\email{amine.marrakchi@math.u-psud.fr}

\thanks{CH and AM are supported by ERC Starting Grant GAN 637601}

\author{Peter Verraedt}
\email{peter.verraedt@math.u-psud.fr}

\thanks{PV is supported by a Ph.D.\ fellowship of the Research Foundation -- Flanders (FWO) and by ERC Starting Grant GAN 637601}



\maketitle


\section{Introduction and statement of the main results}

\subsection*{Introduction} 
The notion of {\em fullness} has always played a central role in the theory of von Neumann algebras. A type ${\rm II_1}$ factor $M$ is \emph{full}, or equivalently does not have property Gamma of Murray and von Neumann \cite{MvN43}, if every uniformly bounded net $(x_i)_{i \in I}$ that is \emph{central}, meaning that $\lim_i \|x_i a - a x_i\|_2 = 0$ for every $a \in M$, must be {\em trivial}, meaning that $\lim_i \|x_i - \tau(x_i)1 \|_2 =0$. Murray and von Neumann \cite{MvN43} used this property to provide the first example of two nonisomorphic factors of type $\II_1$ with separable predual. Indeed, they showed that the free group factor $\rL(\F_2)$ is full, while the unique hyperfinite factor of type $\II_1$ is not. Connes \cite{Co75b} proved that any full factor of type ${\rm II_1}$ with separable predual satisfies the following spectral gap property: there exist $\kappa > 0$ and a family $a_1, \dots, a_k \in M$ such that
\begin{equation}\label{eq-intro-1}
\forall x \in M, \quad \|x - \tau(x)1\|_2^2 \leq \kappa \sum_{j = 1}^k \|xa_j - a_j x\|_2^2.
\end{equation}
The above spectral gap property for full factors of type ${\rm II_1}$ with separable predual was a crucial tool in the proof of the uniqueness of the amenable factor of type ${\rm II_1}$ with separable predual \cite{Co75b}. The inequality in \eqref{eq-intro-1} also implies that the tensor product of two full  factors of type $\II_1$ is still full.

The notion of fullness for type ${\rm III}$ factors as well as the $\tau$-invariant were introduced by Connes. Following \cite{Co74}, we say that a factor $M$ is {\em full} if every uniformly bounded net $(x_i)_{i \in I}$ in $M$ that is {\em centralizing}, meaning that $\lim_i \|x_i \varphi - \varphi x_i\| = 0$ for all $\varphi \in M_\ast$, must be {\em trivial}, meaning that there exists a bounded net $(\lambda_i)_{i \in I}$ in $\C$ such that $x_i - \lambda_i 1 \to 0$ strongly as $i \to \infty$. 
By \cite[Theorem 3.1]{Co74}, for any full factor $M$, the subgroup of inner automorphisms $\Inn(N)$ is closed in the group of automorphisms $\Aut(M)$ and hence the quotient group $\Out(M) = \Aut(M) / \Inn(M)$ inherits a structure of complete topological group. Connes' invariant $\tau(M)$ is then defined as the weakest topology on $\R$ that makes the canonical modular homomorphism $\R \to \Out(M)$ continuous. 

 The motivation behind the present work comes from the following two fundamental open questions:

\begin{itemize}

\item [(Q$1$)] If $M_1$ and $M_2$ are full factors of type $\III$, is the tensor product $M_1 \ovt M_2$ also full?

\item [(Q$2$)] In that case, is it possible to compute $\tau(M_1 \ovt M_2)$ in terms of $\tau(M_1)$ and $\tau(M_2)$?

\end{itemize} 

The absence of a faithful normal tracial state on full factors of type ${\rm III}$ makes their study peculiarly difficult. In particular, a good type $\III$ analogue of Connes' spectral gap theorem is missing. Recent progress in this direction was made by the second named author in \cite[Theorem A]{Ma16}, but questions (Q$1$) and (Q$2$) remained out of reach.

\subsection*{Statement of the main results} The main novelty of the present paper is an enhanced  spectral gap property for full factors of type ${\rm III}$, which relies on \cite[Theorem A]{Ma16}. More precisely, we prove in Theorem \ref{strong-gap} below that for any full factor $M$ of type ${\rm III}$, there exist a state $\varphi \in M_\ast$, some constant $\kappa > 0$ and a family $a_1, \dots, a_k \in M$ satisfying $a_j \varphi a_j^* \leq \varphi$ for every $j \in \{1, \dots, k\}$ such that
\begin{equation}\label{eq-intro-2}
\forall x \in M, \quad \| x-\varphi(x)1 \|_\varphi^{2} \leq \kappa \left( \sum_{j = 1}^k \| xa_j-a_jx \|_\varphi^{2} + \inf_{ \lambda \in \R_+} \| x \xi_\varphi -\lambda \xi_\varphi x \|^{2} \right).
\end{equation}
Observe that compared to Connes' spectral gap theorem for full factors of type ${\rm II_1}$ as in \eqref{eq-intro-1}, there is an additional term in \eqref{eq-intro-2} given by $\inf_{ \lambda \in \R_+} \| x \xi_\varphi -\lambda \xi_\varphi x \|^{2}$, where $\xi_\varphi \in \rL^{2}(M)$ is the vector implementing $\varphi$ in the standard form of $M$. Nevertheless, \eqref{eq-intro-2} is robust enough to pass to tensor products $M \ovt N$, where $N$ is any von Neumann algebra. This allows us to show that for any full factor $M$ and any von Neumann algebra $N$, every centralizing net in the tensor product $M \ovt N$ must asymptotically lie in $N$.

We also prove in Theorem \ref{strong-gap} below an enhanced spectral gap property for the outer automorphism group of a full factor, which relies on \cite[Lemma 5.3]{Ma16} (see \cite[Corollary 5]{Jo81} for the analogous result in the type ${\rm II_1}$ case). This, in turn, allows us to show that for any full factor $M$ and any von Neumann algebra $N$, the natural group homomorphism $\Out(M) \times \Out(N) \to \Out(M \ovt N)$ is a homeomorphism onto its range.

More precisely, our first main result is:

\begin{letterthm}\label{main-thm-full-tensor-products}
Let $M$ be any full factor and $N$ any von Neumann algebra. Then for every uniformly bounded centralizing net $(x_i)_{i \in I}$ in $M \ovt N$, we can find a uniformly bounded centralizing net $(y_i)_{i \in I}$ in $N$ such that $x_i- y_i \to 0$ strongly as $i \to \infty$.

 Moreover, the natural group homomorphism $\Out(M) \times \Out(N) \to \Out(M \ovt N)$ is a homeomorphism onto its range.

\end{letterthm}

Theorem \ref{main-thm-full-tensor-products} above, applied to the case where $N$ itself is a full factor, provides a positive answer to questions (Q$1$) and (Q$2$).

\begin{lettercor}\label{cor-full-tensor-products}
Let $M_1$ and $M_2$ be any full factors. Then $M_1 \ovt M_2$ is a full factor.

Moreover, for any net $(t_i)_{i \in I}$ in $\R$, we have
$$t_i \to 0 \text{ w.r.t. } \tau(M_1 \ovt M_2) \quad \text{if and only if} \quad  \; t_i \to 0 \text{ w.r.t. } \tau(M_1) \text{ and } \tau(M_2).$$

\end{lettercor}

As we mentioned before, the inequality \eqref{eq-intro-2} obtained in Theorem \ref{strong-gap} provides an analogue of Connes' spectral gap theorem for full factors of type $\III$. Although \eqref{eq-intro-2} contains the additional term $\inf_{ \lambda \in \R_+} \| x \xi_\varphi -\lambda \xi_\varphi x \|^{2}$, we conjecture that it can be removed, to obtain the following optimal spectral gap property.

\begin{letterconj}[Strong fullness conjecture] \label{conj-gap}
Let $M$ be any $\sigma$-finite full factor of type ${\rm III}$. Then there exist a faithful state $\varphi \in M_\ast$, some constant $\kappa > 0$ and a family $a_1,\dots,a_k \in M$ satisfying $a_j \varphi a_j^* \leq \varphi$ for every $j \in \{ 1, \dots, k \}$ such that
\begin{equation*}
\forall x \in M, \quad \| x -\varphi(x) 1 \|_\varphi^2 \leq \kappa \sum_{j = 1}^k \| x a_j - a_j x \|_\varphi^2.
\end{equation*}
\end{letterconj}

As we show in Proposition \ref{prop-conjecture}, a consequence of this conjecture would be that a factor $M$ is full if and only if $\mathbf{K}(\rL^{2}(M)) \subset \rC^*(M, JMJ)$, as in \cite{Co75b}. By \cite[Theorem 11]{Ba93}, the conjecture holds for every free product factor $(M, \varphi) = (M_1, \varphi_1) \ast (M_2, \varphi_2)$ for which there exist $u, v \in \mathcal U((M_1)_{\varphi_1})$ and $w \in \mathcal U((M_2)_{\varphi_2})$ such that $\varphi_1(u) = \varphi_1(v) = \varphi_1(u^*v) = \varphi_2(w) = 0$. By \cite[Lemma A.1]{Va04}, it holds for every free Araki--Woods factor \cite{Sh96}. By the proof of \cite[Lemma 2.7]{VV14}, it also holds for plain Bernoulli crossed products of non-amenable groups.

Our next main result shows that, besides the aforementioned particular classes, the conjecture holds for {\em all} full factors of type ${\rm III}$ possessing an {\em almost periodic} state. In this case, we can even choose the family $a_1,\dots,a_k$ in the centralizer $M_\varphi$ of any almost periodic state $\varphi$ for which $M_\vphi$ is a factor.

\begin{letterthm}\label{main-thm-almost-periodic} 
Let $M$ be any full factor with separable predual which possesses an almost periodic faithful normal state. Then for every almost periodic faithful normal state $\varphi \in M_\ast$ whose centralizer $M_\varphi$ is a factor, there exist a constant $\kappa > 0$ and a family $a_1, \dots, a_k \in M_\varphi$ such that 
$$\forall x \in M, \quad \|x - \varphi(x)1\|_\varphi^2 \leq \kappa \sum_{j = 1}^k \|x a_j - a_j x\|_\varphi^2.$$
\end{letterthm}
The proof of Theorem \ref{main-thm-almost-periodic} is based on a semifinite version of \cite[Corollary 5]{Jo81} (see Theorem \ref{crossed-product-finite}) applied to the discrete decomposition $M = N \rtimes \Gamma$, where $N$ is the discrete core of $M$ and $\Gamma = \Sd(M)$.  For this, we exploit the fact that the discrete core $N$ is full (see \cite[Proposition 5]{TU14}) and we prove that the image of $\Gamma$ in $\Out(N)$ is discrete (see Theorem \ref{thm-crossed-product}).

Finally, our last main result is an application of Theorem \ref{main-thm-full-tensor-products} to obtain a \emph{unique McDuff decomposition} result. Following \cite{McD69, Co75a}, we say that a factor $\mathcal{M}$ with separable predual is {\em McDuff} if it absorbs tensorially the hyperfinite type ${\rm II_1}$ factor $R$, that is, $\mathcal M \cong \mathcal M \ovt R$. We introduce the following terminology.

\begin{definition}
Let $\mathcal M$ be any McDuff factor with separable predual. 

\begin{itemize}

\item We say that $\mathcal M$ admits a {\em McDuff decomposition} if there exist a non-McDuff factor $M$ and a non-type ${\rm I}$ amenable factor $P$ such that $\mathcal  M = M \ovt P$.

\item We say that $\mathcal M$ has a {\em unique} McDuff decomposition if the above decomposition $\mathcal{M}=M \ovt P$ is unique up to {\em stable unitary conjugacy} (see Definition \ref{stable_conj}).

\end{itemize} 
\end{definition}

The class of McDuff factors with separable predual that admit a unique McDuff decomposition is well understood in the type $\II_1$ case. Indeed, by a theorem of Popa (see \cite[Theorem 5.1]{Po06}), if $M$ is a full factor of type $\mathrm{II}_1$ with separable predual and $R$ is the hyperfinite type $\mathrm{II}_1$ factor, then the tensor product factor $M \ovt R$ has a unique McDuff decomposition. Conversely, by a theorem of Hoff (see \cite[Theorem B]{Ho15}), if $\mathcal M$ is a McDuff factor of type $\II_1$ with separable predual and with a unique McDuff decomposition $\mathcal M = M \ovt R$, then its non-McDuff part $M$ must be full. 

We extend this characterization of McDuff factors with a unique McDuff decomposition to type ${\rm III}$ factors.

\begin{letterthm}\label{main-thm-mcduff}
Let $\mathcal{M}$ be any McDuff factor with separable predual. The following conditions are equivalent:
\begin{itemize}

\item [$(\rm i)$] $\mathcal{M}=M \ovt P$, where $M$ is a full factor and $P$ is a non-type ${\rm I}$ amenable factor. 

\item [$(\rm ii)$] $\mathcal{M}$ has a unique McDuff decomposition.

\end{itemize}

\end{letterthm}

The proof of the implication $(\rm i) \Rightarrow (\rm ii)$ generalizes the proof of \cite[Theorem 5.1]{Po06} (see also \cite[Theorem 2.3]{Ho06}), which is based on Popa's deformation/rigidity argument and Connes' spectral gap characterization of full factors of type ${\rm II_1}$ as in \eqref{eq-intro-1}. In the type $\III$ case, we use instead the result obtained in Theorem \ref{ultraproduct_full}, and we exploit the recent generalization of Popa's intertwining theory to arbitrary von Neumann algebras (see \cite{HI15, BH16}) to prove the unique McDuff decomposition. A key part in our proof is to reduce to the case when $P = R_\infty$ is the unique amenable factor of type ${\rm III_1}$ and to exploit the fact that $R_\infty$ admits a faithful normal state whose centralizer is irreducible inside $R_\infty$. The proof of the implication $(\rm ii) \Rightarrow (\rm i)$ is an adaptation of the one of \cite[Theorem B]{Ho15} using \cite[Theorem 3.1]{HU15}.

\subsection*{Acknowledgments} We thank Hiroshi Ando, R\'emi Boutonnet and Yoshimichi Ueda for their useful remarks.

\tableofcontents

\section{Preliminaries} 
\subsection*{Basic notations}
Let $M$ be any von Neumann algebra. We denote by $M_\ast$ its predual, by $\mathcal U(M)$ its group of unitaries and by $\mathcal Z(M)$ its center. The uniform norm on $M$ is denoted by $\| \cdot \|_\infty$ and the unit ball of $M$ with respect to the uniform norm $\|\cdot\|_\infty$ is denoted by $\Ball(M)$. If $\varphi \in M_*^+$ is a positive functional, we put $\| x \|_\varphi=\varphi(x^*x)^{1/2}$ for all $x \in M$. 

\subsection*{Standard form}  Let $M$ be any von Neumann algebra. We denote by $(M, \rL^2(M), J, \rL^2(M)_+)$ the standard form of $M$. Recall that $\rL^2(M)$ is naturally endowed with the structure of a $M$-$M$-bimodule: we will simply write $x \xi y = x Jy^*J \xi$ for all $x, y \in M$ and all $\xi \in \rL^2(M)$. The vector $J \xi$ will be also simply denoted by $\xi^{*}$ so that $(x\xi)^{*}=\xi^{*}x^{*}$. For every $\xi \in \rL^2(M)$ we denote by $|\xi| \in \rL^2(M)_+$ its positive part. If $\xi \in \rL^2(M)_+$ we denote by $\xi^2 \in M_*^+$ the positive functional on $M$ defined by $x \mapsto \langle x \xi, \xi \rangle$. For any positive functional $\varphi \in M_\ast^+$, there exists a unique $\xi \in \rL^2(M)_+$ such that $\xi^2=\varphi$. We denote it by $\xi_\varphi \in \rL^2(M)_+$ (note that $\xi_\varphi$ is denoted $\varphi^{1/2}$ in \cite{Ma16}). We then have $\|x\|_\varphi = \|x \xi_\varphi\|$ for all $x \in M$. The following remark will be very useful.
\begin{rem}\label{rem-standard-form}
For every positive functional $\varphi \in M_\ast$, we have 
$$\left\{\eta \in \rL^2(M) : |\eta|^2 \leq \varphi \right\}  =  \Ball(M) \xi_\varphi$$
(see the discussion after \cite[Lemma 3.2]{Ma16} for further details). 
Moreover, by the polarization identity, we have that
$$\left\{\eta \in \rL^2(M) : \exists \lambda \in \R_+, \;  |\eta|^2 \leq \lambda\varphi  \text{ and }  |\eta^{*}|^2 \leq \lambda\varphi  \right\} $$ is linearly spanned by $\left\{\eta \in \rL^2(M)_+ : \eta^2 \leq \varphi \right\}$.
\end{rem}

\subsection*{Relative modular operators}
Take any pair of positive linear functionals $\varphi, \psi \in M_*^+$ with supports $p=\mathrm{supp}(\varphi)$ and $q=\mathrm{supp}(\psi)$. The antilinear operator densely defined on $\rL^2(M)p^{\perp} \oplus M \xi_\varphi \subset \rL^2(M)$ by the formula $\eta \oplus x\xi_\varphi \mapsto x^*\xi_\psi$ for $x \in Mp$ is closable and its closure $S$ has polar decomposition $S=J_{\psi,\vphi}\Delta^{1/2}_{\psi,\varphi}$, where $J_{\psi,\vphi}=JpJqJ$ and $\Delta_{\psi,\varphi}=S^*S$ is the \emph{relative modular operator} of $\psi$ with respect to $\varphi$. It is a closed positive densely defined operator on $\rL^2(M)$ supported on $qJpJ$. Moreover, $\Delta_{\psi,\varphi}^{1/2}$ is the closure of the closable operator densely defined on $\rL^2(M)p^{\perp} \oplus M\xi_\varphi \subset \rL^2(M)$ by $\eta \oplus x\xi_\varphi \mapsto \xi_\psi x $ for $x \in Mp$. When $\varphi=\psi$, we will simply denote $\Delta_{\varphi,\varphi}$ by $\Delta_\varphi$.

If $M$ and $N$ are two von Neumann algebras, then we have a natural identification $\rL^2(M \ovt N) = \rL^2(M) \ovt \rL^2(N)$ which identifies $\xi_{\varphi \otimes \psi}$ with $\xi_\varphi \otimes \xi_\psi$ for any pair of positive functionals $\varphi \in M_*^+$ and $\psi \in N_*^+$. Moreover, by this identification, we have $\Delta_{\varphi_2 \otimes \psi_2,\varphi_1 \otimes \psi_1}=\Delta_{\varphi_2,\varphi_1} \otimes \Delta_{\psi_2, \psi_1}$ as closed densely defined operators. Here, for any pair of closed densely defined operators $S,T$ on two Hilbert spaces $H$ and $K$ with domains $\mathcal{D}(S)$ and $\mathcal{D}(T)$, the operator $S \otimes T : \mathcal D(S \otimes T) \to H \otimes K$ denotes the closure of the closable densely defined operator $S \odot T : \mathcal D(S) \odot \mathcal D(T) \to H \otimes K : \xi \otimes \zeta \mapsto S(\xi) \otimes T(\zeta)$.

\subsection*{Topological groups associated to a von Neumann algebra}
A \emph{topological group} is a group $G$ equipped with a topology making the map 
$ (g,h) \in G \times G \mapsto gh^{-1}$ continuous. If $H$ is a normal subgroup of $G$, then  $G/H$ is a topological group with respect to the quotient topology (not necessarily Hausdorff). A topological group $G$ is said to be \emph{complete} if it is Hausdorff and complete with respect to the uniform structure generated by the following sets
$$ U_\mathcal{V}=\{ (g,h) \in G \times G \mid gh^{-1} \in \mathcal{V} \text{ and } g^{-1}h \in \mathcal{V} \}, $$
where $\mathcal{V}$ runs over the neighborhoods of $1$ in $G$. If $G$ is complete and $H$ is a subgroup of $G$, then $H$ is complete if and only if it is closed in $G$. If moreover $H$ is normal, then $G/H$ is also complete. We say that a complete topological group is Polish if it is separable and completely metrizable.

Let $M$ be a von Neumann algebra. Then the restriction of the strong topology and the $*$-strong topology coincide on $\mathcal{U}(M)$ and they turn $\mathcal{U}(M)$ into a complete topological group. If moreover $M_*$ has separable predual, then $\mathcal{U}(M)$ is Polish.

The group $\Aut(M)$ of all $\ast$-automorphisms of $M$ acts on $M_*$ by $\theta(\varphi)=\varphi \circ \theta^{-1}$ for all $\theta \in \Aut(M)$ and all $\varphi \in M_*$. Following \cite{Co74, Ha73}, the $u$-topology on $\Aut(M)$ is the topology of pointwise norm convergence on $M_*$, meaning that a net $(\theta_i)_{i \in I}$ in $\Aut(M)$ converges to the identity $\id_M$ in the $u$-topology if and only if for all $\varphi \in M_*$ we have $\| \theta_i(\varphi) -\varphi \| \to 0$ as $i \to \infty$. This turns $\Aut(M)$ into a complete topological group. When $M_*$ is separable, $\Aut(M)$ is Polish. Since the standard form of $M$ is unique, the group $\Aut(M)$ also acts naturally on $\rL^2(M)$ and we have $\theta(\xi_\varphi)=\xi_{\theta(\varphi)}$ for every $\varphi \in M_*^+$. Then the $u$-topology is also the topology of pointwise norm convergence on $\rL^2(M)$.

We denote by $\Ad : \mathcal{U}(M) \rightarrow \Aut(M)$ the continuous homomorphism which sends a unitary $u$ to the corresponding inner automorphism $\Ad(u)$. We denote by $\Inn(M)$ the image of $\Ad$. We denote by $\Out(M)=\Aut(M)/\Inn(M)$ the quotient group. Let $\pi_M : \Aut(M) \to \Out(M)$ be the quotient map. A net $(\theta_i)_{i \in I}$ in $\Aut(M)$ satisfies $\pi_M(\theta_i) \to 1$ in $\Out(M)$ as $i \to \infty$ if and only if there exists a net $(u_i)_{i \in I}$ in $\mathcal{U}(M)$ such that $\Ad(u_i) \circ \theta_i \to \id_M$ as $i \to \infty$. 

Note that a net of unitaries $(u_i)_{i \in I}$ in $\mathcal{U}(M)$ is centralizing in $M$ if and only if $\Ad(u_i) \to \id_M$ as $i \to \infty$. Hence, when $M$ is a full factor, the definitions imply that the map $\Ad : \mathcal{U}(M) \rightarrow \Aut(M)$ is open on its range. Therefore, $\Inn(M)$ is isomorphic as a topological group to the quotient $\mathcal{U}(M)/\{ z \in \C \mid |z|=1 \}$. In particular, $\Inn(M)$ is complete hence closed in $\Aut(M)$ and $\Out(M)$ is a complete topological group.

Finally, we recall the following useful fact (see e.g.\ \cite[Proposition 2.8]{Co74}).
\begin{rem}\label{rem-convergence}
Let $(\theta_i)_{i \in I}$ be any net in $\Aut(M)$ such that the following two properties are satisfied:
\begin{itemize}

\item There exists a strongly dense subset $S \subset M$ such that $\theta_i(x) \to x$ strongly as $i \to \infty$ for every $x \in S$.

\item There exists a faithful state $\varphi \in M_\ast$ such that $\lim_i\| \theta_i(\varphi) - \varphi\| = 0$.
\end{itemize}
Then $\theta_i \to \id_M$ in $\Aut(M)$ as $i \to \infty$. 
\end{rem}

\subsection*{Ultraproducts von Neumann algebras}
Let $M$ be any $\sigma$-finite von Neumann algebra. Let $I$ be any nonempty directed set and $\omega$ any {\em cofinal} ultrafilter on $I$, i.e.\ $\{i \in I : i \geq i_0\} \in \omega$ for every $i_0 \in I$. When $I = \N$, $\omega$ is cofinal if and only if $\omega$ is {\em nonprincipal}, i.e.\ $\omega \in \beta(\N) \setminus \N$. Define
\begin{align*}
\mathcal I_\omega(M) &= \left\{ (x_i)_i \in \ell^\infty(I, M) \mid x_i \to 0\ \ast\text{-strongly as } i \to \omega \right\} \\
\mathfrak M^\omega(M) &= \left \{ (x_i)_i \in \ell^\infty(I, M) \mid  (x_i)_i \, \mathcal I_\omega(M) \subset \mathcal I_\omega(M) \text{ and } \mathcal I_\omega(M) \, (x_i)_i \subset \mathcal I_\omega(M)\right\}.
\end{align*}
The multiplier algebra $\mathfrak M^\omega(M)$ is a $\rC^*$-algebra and $\mathcal I_\omega(M) \subset \mathfrak M^\omega(M)$ is a norm closed two-sided ideal. Following \cite[\S 5.1]{Oc85}, the quotient $\rC^{*}$-algebra $M^\omega := \mathfrak M^\omega(M) / \mathcal I_\omega(M)$ is a von Neumann algebra, known as the {\em Ocneanu ultraproduct} of $M$. We denote the image of $(x_i)_i \in \mathfrak M^\omega(M)$ by $(x_i)^\omega \in M^\omega$. Throughout this paper, we will use the notation from \cite{AH12} for ultraproducts.

We say that a von Neumann subalgebra $P \subset M$ is {\em with expectation} if there exists a faithful normal conditional expectation $\rE_P : M \to P$. In that case, for every nonempty directed set $I$ and every cofinal ultrafilter $\omega$ on $I$, $P^\omega$ is a natural von Neumann subalgebra of $M^\omega$ and the inclusion $P^\omega \subset M^\omega$ is naturally endowed with a faithful normal conditional expectation $\rE_{P^\omega} : M^\omega \to P^\omega$ satisfying $\rE_{P^\omega}((x_i)^\omega) = (\rE_P(x_i))^\omega$ for every $(x_i)^\omega \in M^\omega$ (see the discussion in \cite[Section 2]{HI15} for further details).

\section{Strengthening the spectral gap property for full factors}

We start by recalling the following spectral gap property for full factors and their outer automorphism groups obtained in \cite[Theorem 4.4]{Ma16} and \cite[Lemma 5.3]{Ma16}.

\begin{thm}[{\cite{Ma16}}] \label{gap}
Let $M$ be any full factor. Let $\mathcal V$ be any neighborhood of $1$ in $\Out(M)$. Then there exist a state $\varphi \in M_\ast$ and a family $\xi_1, \dots, \xi_k \in \rL^2(M)_+$ with $\xi_j^2 \leq \varphi$ for every $j \in \{1,\dots, k \}$ such that for all $x \in M$ we have 
\begin{equation}\label{eq-spectral-gap}
\| x-\varphi(x)1 \|_\varphi^2 \leq \sum_{j = 1}^k \| x \xi_j-\xi_j x \|^2
\end{equation}
and for all $x \in M$ and all $\theta \in \Aut(M) \setminus \pi_M^{-1}(\mathcal V)$ we have 
\begin{equation}\label{eq-spectral-gap-out}
\| x \|_\varphi^2 \leq  \sum_{j = 1}^k \| x \xi_j-\theta(\xi_j) x \|^2.
\end{equation}
Moreover, if $M$ is of type $\mathrm{III}$ and $p$ is any nonzero $\sigma$-finite projection, we can choose $\varphi$ such that $\mathrm{supp}(\varphi)=p$. If $M$ is semifinite and $p$ is any nonzero finite projection, we can take $\varphi=p\tau p$ where $\tau$ is the unique semifinite normal trace of $M$ with $\tau(p)=1$.
\end{thm}

Theorem \ref{strong-gap} below is a strengthening of Theorem \ref{gap}. It will be crucial in the proof of Theorem \ref{main-thm-full-tensor-products} and Corollary \ref{cor-full-tensor-products}.

\begin{thm} \label{strong-gap}
Let $M$ be any full factor. Let $\mathcal V$ be any neighborhood of $1$ in $\Out(M)$. Take a state $\varphi \in M_*$ and a family $\xi_1, \dots, \xi_k \in \rL^2(M)_+$ with $\xi_j^2 \leq \varphi$ for every $j \in \{1,\dots, k \}$ such that the conclusion of Theorem \ref{gap} holds. For every $j \in \{1, \dots, k\}$, write $\xi_j=a_j \xi_\varphi=\xi_\varphi a_j^{*}$ with $a_j \in \Ball(M)$. Then there exists a constant $\kappa > 0$ such that for all $x \in M$ we have
\begin{equation}\label{eq-spectral-gap-bis}
\| x-\varphi(x)1 \|_\varphi^{2} \leq \kappa \left( \sum_{j = 1}^k \| xa_j-a_jx \|_\varphi^{2} + \inf_{ \lambda \in \R_+} \| x \xi_\varphi -\lambda \xi_\varphi x \|^{2} \right)
\end{equation}
and for all $x \in M$ and all $\theta \in \Aut(M) \setminus \pi_M^{-1}(\mathcal V)$ we have
\begin{equation}\label{eq-spectral-gap-out-bis}
\| x \|_\varphi^{2} \leq \kappa \left( \sum_{j = 1}^k \| xa_j - \theta(a_j) x \|_\varphi^{2} + \inf_{ \lambda \in \R_+} \left\| x \xi_\varphi -\lambda \theta(\xi_\varphi) x \right\|^{2} \right).
\end{equation}
\end{thm}

\begin{proof}[Proof of Theorem \ref{strong-gap}]
{\bf Proof of the first item \eqref{eq-spectral-gap-bis}.} By homogeneity, it is sufficient to prove that for all sequences $x_n \in M$ and $\lambda_n \in \R_+$ $(n \in \N)$ such that $\lim_n \| x_n a_j- a_j x_n \|_\varphi = 0$ for every $j \in \{1, \dots, k\}$ and $\lim_n \| x_n \xi_\varphi -\lambda_n \xi_\varphi x_n \| = 0$, we have that $\lim_n \| x_n-\varphi(x_n) 1 \|_\varphi = 0$. We will do this by using the same technique as in the proof of \cite[Lemma 5.3]{Ma16}.

Put $\mu_n = \|x_n\|_\varphi \in \R_+$ for every $n \in \N$. Observe that the sequence $(\mu_n)_{n \in \N}$ may be unbounded. Since $\lim_n \| x_n \xi_\varphi -\lambda_n \xi_\varphi x_n \| = 0$, we have that $\lim_n  \left| \mu_n - \lambda_n \| x_n^{*}\|_\varphi \right|= 0$. For every $n \in \N$, write $x_n= u_n |x_n| = u_n |x_n| u_n^* \, u_n = |x_n^*|u_n$ for the polar decomposition of $x_n \in M$. For every $j \in \{1, \dots, k\}$ and every $n \in \N$, we have
\begin{align*}
\| x_n \xi_j-\lambda_n \xi_jx_n \| &= \| (x_n a_j - a_j x_n)\xi_\varphi  + a_j( x_n \xi_\varphi - \lambda_n \xi_\varphi x_n) \| \\
&\leq \| x_n a_j-a_jx_n \|_\varphi + \| x_n \xi_\varphi -\lambda_n \xi_\varphi x_n \|.
\end{align*}
Thus, we obtain $\lim_n \| x_n \xi_j-\lambda_n \xi_j x_n \| = 0$ for every $j \in \{1, \dots, k\}$. By \cite[Proposition 3.1 (3)]{Ma16} applied to ``$\xi = \xi_j \in \rL^2(M)_+$" and ``$\eta = \lambda_n \xi_j \in \rL^2(M)_+$", we also have
$$ 2\| x_n \xi_j-\lambda_n \xi_j x_n \|^2 \geq \|  |x_n| \xi_j-\xi_j |x_n|   \|^2+ \| \lambda_n (|x_n^*| \xi_j - \xi_j |x_n^*|) \|^2.$$
Hence we obtain $\lim_n \|  |x_n| \xi_j-\xi_j |x_n|  \| =  0$ and $\lim_n \lambda_n \|  |x_n^*| \xi_j - \xi_j |x_n^*|  \| = 0$ for every $j \in \{1, \dots, k\}$. Then \eqref{eq-spectral-gap} yields $\lim_n \|  |x_n|-\varphi(|x_n|)1  \|_\varphi = 0$ and $\lim_n \lambda_n \|  |x_n^{*}|-\varphi(|x_n^{*}|)1 \|_\varphi =~0$.

Since $ \| |x_n| \|_\varphi  = \|x_n \|_\varphi = \mu_n$ for every $n \in \N$ and since $\lim_n \|  |x_n|-\varphi(|x_n|)1  \|_\varphi = 0$, we have $\lim_n \left|\mu_n - \varphi (|x_n|) \right| = 0$ and hence we obtain $\lim_n \|  |x_n|- \mu_n 1 \|_\varphi = 0$. Likewise, since $\lim_n  \left| \mu_n - \lambda_n \|  |x_n^{*}|  \|_\varphi \right| = \lim_n  \left| \mu_n - \lambda_n \| x_n^{*}\|_\varphi \right| = 0$ and since $\lim_n \| \lambda_n |x_n^{*}| - \lambda_n \varphi(|x_n^{*}|)1  \|_\varphi = 0$, we have $\lim_n | \mu_n - \lambda_n \varphi(|x_n^{*}|)| = 0$ and hence we obtain $\lim_n \|  \lambda_n |x_n^{*}|-\mu_n 1 \|_\varphi = 0$. Since $\lim_n \| |x_n|\xi_\varphi - \mu_n\xi_\varphi  \| = 0$, by multiplying by $u_n$ on the left, we obtain  $\lim_n \|x_n\xi_\varphi - \mu_n u_n\xi_\varphi \| = 0$. Likewise, since $\lim_n \| \lambda_n |x_n^*| \xi_\varphi - \mu_n \xi_\varphi  \| = \lim_n \| \lambda_n \xi_\varphi  |x_n^*| - \mu_n \xi_\varphi  \| = 0$ (using the operator $J$), by multiplying by $u_n$ on the right, we obtain $\lim_n \|  \lambda_n \xi_\varphi x_n - \mu_n\xi_\varphi u_n  \| = 0$. Since $\lim_n \| x_n \xi_\varphi -\lambda_n \xi_\varphi x_n \| =  0$, this yields $\lim_n \mu_n \| u_n \xi_\varphi - \xi_\varphi u_n \| = 0$.

Recall that $a_j \in \Ball(M)$ for all $j \in \{1, \dots, k\}$. Then we obtain 
\begin{align*}
\limsup_n \mu_n \|u_n \xi_j - \xi_j u_n\| &= \limsup_n  \| \mu_n u_n a_j \xi_\varphi - \mu_n a_j \xi_\varphi u_n\| \\
&= \limsup_n  \|\mu_n u_n  \xi_\varphi a_j^* - \mu_n a_j \xi_\varphi u_n\| \quad (\text{since } a_j \xi_\varphi = \xi_\varphi a_j^*) \\
&\leq \limsup_n  \|\mu_nu_n  \xi_\varphi a_j^* - \mu_n a_j u_n \xi_\varphi \|  \quad (\text{since } \lim_n \mu_n \|u_n \xi_\varphi - \xi_\varphi u_n\| = 0) \\
&\leq \limsup_n  \|x_n  \xi_\varphi a_j^* - a_j x_n \xi_\varphi \|  \quad (\text{since } \lim_n \|x_n \xi_\varphi -  \mu_n u_n\xi_\varphi \| = 0)\\
&\leq \limsup_n  \|x_n  a_j \xi_\varphi - a_j x_n \xi_\varphi \| \quad (\text{since } a_j \xi_\varphi = \xi_\varphi a_j^*) \\
&= 0.
\end{align*}
This implies that $\lim_n  \mu_n\|u_n \xi_j - \xi_j u_n\| = 0$ for all $j \in \{1, \dots, k\}$. Then \eqref{eq-spectral-gap} implies that $\lim_n \|\mu_n u_n - \varphi(\mu_n u_n) 1\|_\varphi = 0$. Since $\lim_n \|x_n - \mu_n u_n\|_\varphi = 0$, Cauchy-Schwarz inequality yields $\lim_n |\varphi(x_n - \mu_nu_n)| = 0$. Therefore, we finally obtain $\lim_n \|x_n - \varphi(x_n) 1\|_\varphi = 0$.

{ \bf Proof of the second item \eqref{eq-spectral-gap-out-bis}.} We follow the same pattern as before. By homogeneity, it is sufficient to prove that for all sequences $x_n \in M$, $\lambda_n \in \R_+$ and $\theta_n \in \Aut(M) \setminus \pi_M^{-1}(\mathcal V)$ $(n \in \N)$ such that $\lim_n \| x_n a_j- \theta_n(a_j) x_n \|_\varphi = 0$ for every $j \in \{1, \dots, k\}$ and $\lim_n \| x_n \xi_\varphi -\lambda_n \theta_n(\xi_\varphi) x_n \| = 0$, we have that $\lim_n \| x_n \|_\varphi = 0$. 

Put $\mu_n = \|x_n\|_\varphi \in \R_+$ for every $n \in \N$. Since $\lim_n \| x_n \xi_\varphi -\lambda_n \theta_n(\xi_\varphi) x_n \| = 0$, we have that $\lim_n  \left| \mu_n - \lambda_n \| \theta_n^{-1}(x_n^{*})\|_\varphi \right|= \lim_n  \left| \mu_n - \lambda_n \| \theta_n(\xi_\varphi)x_n\| \right| =  0$.  For every $n \in \N$, write $x_n= u_n |x_n| = u_n |x_n| u_n^* \, u_n=|x_n^*|u_n$ for the polar decomposition of $x_n \in M$. For every $j \in \{1, \dots, k\}$ and every $n \in \N$, we have
\begin{align*}
\| x_n \xi_j-\lambda_n \theta_n(\xi_j) x_n \| &= \| (x_n a_j - \theta_n(a_j) x_n)\xi_\varphi  + \theta_n(a_j)( x_n \xi_\varphi - \lambda_n \theta_n(\xi_\varphi) x_n) \| \\
&\leq \| x_n a_j-\theta_n(a_j)x_n \|_\varphi + \| x_n \xi_\varphi -\lambda_n \theta_n(\xi_\varphi) x_n \|.
\end{align*}
Thus, we obtain $\lim_n \| x_n \xi_j-\lambda_n \theta_n(\xi_j) x_n \| = 0$ for every $j \in \{1, \dots, k\}$. By \cite[Proposition 3.1 (3)]{Ma16} applied to ``$\xi = \xi_j \in \rL^2(M)_+$" and ``$\eta = \lambda_n \theta_n(\xi_j) \in \rL^2(M)_+$", we also have
$$ 2\| x_n \xi_j-\lambda_n \theta_n(\xi_j) x_n \|^2 \geq \|  |x_n| \xi_j-\xi_j |x_n|  \|^2+ \|  \lambda_n (|x_n^*| \theta_n(\xi_j) - \theta_n(\xi_j) |x_n^*|)  \|^2.$$
Hence we obtain $\lim_n \|  |x_n| \xi_j-\xi_j |x_n|  \| =  0$ and $\lim_n \lambda_n \| \, \theta_n^{-1}(|x_n^*|) \xi_j - \xi_j \theta_n^{-1}(|x_n^*|)  \| = 0$ for every $j \in \{1, \dots, k\}$. Then \eqref{eq-spectral-gap} implies that $\lim_n \|  |x_n|-\varphi(|x_n|)1  \|_\varphi = 0$ and $\lim_n \lambda_n \| \theta_n^{-1}(|x_n^*|) -\varphi(\theta_n^{-1}(|x_n^*|))1   \|_\varphi = 0$. 

Since $ \| |x_n| \|_\varphi  = \|x_n \|_\varphi = \mu_n$ for every $n \in \N$ and since $\lim_n \|  |x_n|-\varphi(|x_n|)1  \|_\varphi = 0$, we have $\lim_n \left|\mu_n - \varphi (|x_n|) \right| = 0$ and hence we obtain $\lim_n \|  |x_n| \xi_\varphi- \mu_n \xi_\varphi \| = 0$. Likewise, since $\lim_n  \left| \mu_n - \lambda_n \|  \theta_n^{-1}(|x_n^{*}|)  \|_\varphi \right| = \lim_n  \left| \mu_n - \lambda_n \| \theta_n^{-1}(x_n^{*})\|_\varphi \right| = 0$ and since $\lim_n \|  \lambda_n \theta_n^{-1}(|x_n^{*}|) - \lambda_n \varphi(\theta_n^{-1}(|x_n^{*}|))1   \|_\varphi = 0$, we have $\lim_n | \mu_n - \lambda_n \varphi(\theta_n^{-1}(|x_n^{*}|))| = 0$. This in turn implies that $\lim_n \|  \lambda_n |x_n^{*}| \theta_n(\xi_\varphi)-\mu_n \theta_n(\xi_\varphi) \| = \lim_n \|  \lambda_n \theta_n^{-1}(|x_n^{*}|)-\mu_n 1 \|_\varphi = 0$. Since $\lim_n \| |x_n|\xi_\varphi - \mu_n\xi_\varphi  \| =0$, by multiplying by $u_n$ on the left, we obtain  $\lim_n \|x_n\xi_\varphi - \mu_n u_n\xi_\varphi \| = 0$. Likewise, since $\lim_n \| \lambda_n |x_n^*| \theta_n(\xi_\varphi) - \mu_n \theta_n(\xi_\varphi)  \| = \lim_n \| \lambda_n \theta_n(\xi_\varphi)  |x_n^*| - \mu_n \theta_n(\xi_\varphi)  \| = 0$ (using the operator $J$), by multiplying by $u_n$ on the right, we obtain $\lim_n \|  \lambda_n \theta_n(\xi_\varphi) x_n - \mu_n\theta_n(\xi_\varphi) u_n  \| = 0$. Since $\lim_n \| x_n \xi_\varphi -\lambda_n \theta_n(\xi_\varphi) x_n \| =  0$, this yields $\lim_n \mu_n \| u_n \xi_\varphi - \theta_n(\xi_\varphi) u_n \| = 0$. 

Recall that $a_j \in \Ball(M)$ for all $j \in \{1, \dots, k\}$. Then we obtain 
\begin{align*}
& \limsup_n \mu_n \|u_n \xi_j - \theta_n(\xi_j) u_n\| \\
&= \limsup_n  \| \mu_n u_n a_j \xi_\varphi - \mu_n \theta_n(a_j) \theta_n(\xi_\varphi) u_n\| \\
&= \limsup_n  \|\mu_n u_n  \xi_\varphi a_j^* - \mu_n \theta_n(a_j) \theta_n(\xi_\varphi) u_n\| \quad (\text{since } a_j \xi_\varphi = \xi_\varphi a_j^*) \\
&\leq \limsup_n  \|\mu_nu_n  \xi_\varphi a_j^* - \mu_n \theta_n(a_j) u_n \xi_\varphi \|  \quad (\text{since } \lim_n \mu_n \|u_n \xi_\varphi - \theta_n(\xi_\varphi) u_n\| = 0) \\
&\leq \limsup_n  \|x_n  \xi_\varphi a_j^* - \theta_n(a_j) x_n \xi_\varphi \|  \quad (\text{since } \lim_n \|x_n \xi_\varphi -  \mu_n u_n\xi_\varphi \| = 0)\\
&\leq \limsup_n  \|x_n  a_j \xi_\varphi - \theta_n(a_j) x_n \xi_\varphi \| \quad (\text{since } a_j \xi_\varphi = \xi_\varphi a_j^*) \\
&= 0.
\end{align*}
This implies that $\lim_n  \mu_n\|u_n \xi_j - \theta_n(\xi_j) u_n\| = 0$ for all $j \in \{1, \dots, k\}$. Then \eqref{eq-spectral-gap-out} implies that $\lim_n \|\mu_n u_n \|_\varphi = 0$. Since $\lim_n \|x_n - \mu_n u_n\|_\varphi = 0$, we obtain $\lim_n \|x_n\|_\varphi = 0$.
\end{proof}

In the semifinite case, we can obtain a much simpler statement.

\begin{thm} \label{spectral-gap-finite}
Let $M$ be any full semifinite factor. Let $p$ be any nonzero finite projection and let $\varphi =p\tau p$ where $\tau$ is the unique semifinite normal trace of $M$ such that $\tau(p)=1$. Let $\mathcal V$ be any neighborhood of $1$ in $\Out(M)$. Then we can find $\kappa > 0$ and $a_1,\dots,a_k \in (pMp)_+$ with $\| a_j \| \leq 1$ such that for all $x \in M$ we have
\begin{equation}\label{eq-spectral-gap-finite}
\| x-\varphi(x)1 \|_\varphi^{2} \leq \kappa \sum_{j = 1}^k \| xa_j-a_jx \|_\varphi^{2}
\end{equation}
and for all $x \in M$ and all $\theta \in \Aut(M) \setminus \pi_M^{-1}(\mathcal V)$ we have
\begin{equation}\label{eq-spectral-gap-out-finite}
\| x \|_\varphi^{2} \leq \kappa \sum_{j = 1}^k \| xa_j - \theta(a_j) x \|_\varphi^{2}. 
\end{equation}

\end{thm}
\begin{proof}
Take $\xi_1,\dots,\xi_k \in \rL^{2}(M)_{+}$ with $\xi_j^2 \leq \varphi$ for every $j \in \{1,\dots, k \}$ satisfying the conclusion of Theorem \ref{gap}. For every $j \in \{1, \dots, k\}$, write $\xi_j=a_j \xi_\varphi \in \rL^2(M)_+$ with $a_j \in \Ball(pMp)$.
Then for all $x \in M$ we have
\[ \| (xp) \xi_j-\xi_j (xp) \|^2 =\| x a_j-a_j x \|_\varphi^2, \]
and therefore
\[ \| x-\varphi(x)1 \|_\varphi^2=\| xp-\varphi(xp)1 \|_\varphi^2 \leq \sum_{j=1}^{k} \| x a_j-a_j x \|_\varphi^2, \]
which is exactly the first item \eqref{eq-spectral-gap-finite}.

For the second item, we will need Theorem \ref{strong-gap}. Let $\tau$ be the unique faithful semifinite normal trace of $M$ such that $\tau p =p \tau=\varphi$. Fix $\theta \in \Aut(M) \setminus \pi_M^{-1}(\mathcal V)$ and let $\lambda=\mathrm{mod}(\theta)^{-1/2}$ so that $\theta(\tau^{1/2})=\lambda^{-1}\tau^{1/2}$, where $\tau^{1/2}$ is viewed as an operator affiliated with the core of $M$ as in \cite{Ma16}. Take $x \in \theta(p)Mp$. Then we have 
\[ \lambda \theta(\xi_\varphi)x=\lambda \theta(\tau^{1/2})x=\tau^{1/2}x=x\tau^{1/2}=x \xi_\varphi \]
hence by \eqref{eq-spectral-gap-out-bis} we have
\[ \| x \|_\varphi^{2} \leq \kappa \sum_{j = 1}^k \| xa_j - \theta(a_j) x \|_\varphi^{2} \]
for all $x \in \theta(p)Mp$. Now, for all $x \in M$ we obtain
\[ \| x \|_\varphi^{2}=\| \theta(p^{\perp})xp \|_\varphi^{2}+\| \theta(p)xp \|_\varphi^{2} \leq \| xp - \theta(p) x \|_\varphi^{2}+\kappa \sum_{j = 1}^k \| xa_j - \theta(a_j) x \|_\varphi^{2}. \]
Finally, by taking $a_{k+1}=p$ and replacing $\kappa$ by $\max(1,\kappa)$ we obtain \eqref{eq-spectral-gap-out-finite}.
\end{proof}

Using Theorem \ref{strong-gap}, we also strengthen \cite[Theorem 5.4]{Ma16}. 

\begin{thm}\label{crossed-product}
Let $M$ be any full factor. Let $\sigma : G \rightarrow \Aut(M)$ be an outer action of a discrete group $G$ such that the image of $\sigma(G)$ in $\Out(M)$ is discrete. Let $\mathcal{V}$ be any neighborhood of $1$ in $\Out(M)$ such that $\sigma( G ) \cap \pi_M^{-1}(\mathcal{V})=\{ \id_M \}$. Choose a state $\varphi \in M_\ast$, a family $a_1,\dots,a_k \in M$ and $\kappa > 0$ which satisfy the conclusion of Theorem \ref{strong-gap}. Then for all $x \in M \rtimes_\sigma G$ we have
$$ \| x- \varphi(x)1 \|_\varphi^{2} \leq \kappa \left( \sum_{j=1}^k \| a_k x-xa_k \|_\varphi^2 + \inf_{ \lambda \in \R_+} \| x \xi_\varphi -\lambda \xi_\varphi x \|^{2} \right), $$
where we used the canonical conditional expectation $\rE: M \rtimes_\sigma G \rightarrow M$ to lift $\varphi$ to a state on $M \rtimes_\sigma G$ and to view $\rL^{2}(M)$ as a substandard form of $\rL^{2}(M \rtimes_\sigma G)$.
\end{thm}
\begin{proof}Take any $x \in M \rtimes_\sigma G$ and let $x^{g}=\rE(u_g^*x) \in M$ for all $g \in G$. Then we have
$$ \| \rE(x)-\varphi(x)1 \|_\varphi^2 \leq \kappa \left( \sum_{j=1}^k \| \rE(x)a_j-a_j \rE(x) \|_\varphi^2 + \inf_{ \lambda \in \R_+} \| \rE(x) \xi_\varphi -\lambda \xi_\varphi \rE(x) \|^{2} \right)$$
and 
$$\forall g \in G \setminus \{1\}, \quad \| x^g \|_\varphi^2 \leq \kappa  \left(\sum_{j=1}^k \| x^ga_j-\sigma_{g^{-1}}(a_j)x^g \|_\varphi^2 + \inf_{ \lambda \in \R_+} \| x^{g} \xi_\varphi -\lambda \sigma_{g^{-1}}(\xi_\varphi) x^{g} \|^{2} \right). $$
Now, since for all $j \in \{1,\dots, k \}$ and all $\lambda \in \R_+$ we have
$$ \| xa_j-a_j x \|_\varphi^2 = \| \rE(x)a_j-a_j \rE(x) \|_\varphi^2+ \sum_{ g \in G \setminus \{1\} } \| x^ga_j-\sigma_{g^{-1}}(a_j)x^g \|_\varphi^2,$$
$$ \| x\xi_\varphi -\lambda \xi_\varphi x \|^2 = \| \rE(x) \xi_\varphi -\lambda \xi_\varphi \rE(x) \|^{2} + \sum_{ g \in G \setminus \{1\} } \| x^{g} \xi_\varphi -\lambda \sigma_{g^{-1}}(\xi_\varphi) x^{g} \|^{2},$$
and since 
$$ \|x-\varphi(x)1 \|_\varphi^2 = \| \rE(x)-\varphi(x)1 \|_\varphi^2 + \|x-\rE(x) \|_\varphi^{2}= \| \rE(x)-\varphi(x)1 \|_\varphi^2 + \sum_{g \in G \setminus \{ 1 \} } \| x^g \|_\varphi^2,$$
we conclude that
\begin{equation*} \|x-\varphi(x)1 \|_\varphi^2 \leq  \kappa \left( \sum_{j=1}^k \| a_k x-xa_k \|_\varphi^2 + \inf_{ \lambda \in \R_+} \| x \xi_\varphi -\lambda \xi_\varphi x \|^{2} \right). \qedhere
\end{equation*}
\end{proof}

With a similar proof, we also obtain a semifinite version that we will need for Theorem \ref{main-thm-almost-periodic}.
\begin{thm}\label{crossed-product-finite}
Let $M$ be any full semifinite factor. Let $\sigma: G \rightarrow \Aut(M)$ be an outer action of a discrete group $G$ such that the image of $\sigma(G)$ in $\Out(M)$ is discrete. Let $\mathcal{V}$ be any neighborhood of $1$ in $\Out(M)$ such that $\sigma( G ) \cap \pi_M^{-1}(\mathcal{V})=\{ \id_M \}$. Take any nonzero finite projection $p \in M$ and let $\varphi=p\tau p$ where $\tau$ is the unique faithful semifinite normal trace of $M$ such that $\tau(p)=1$. Take $\kappa > 0$ and $a_1,\dots,a_k \in M$ as in Theorem \ref{spectral-gap-finite}. Then, for all $x \in M \rtimes_\sigma G$ we have
$$ \| x- \varphi(x)1 \|_\varphi^{2} \leq \kappa \sum_{j=1}^k \| a_k x-xa_k \|_\varphi^2 $$
where we used the canonical conditional expectation $\rE: M \rtimes_\sigma G \rightarrow M$ to lift $\varphi$ to a state on $M \rtimes_\sigma G$.
\end{thm}

We end this section with an application which strengthens \cite[Theorem B]{Ma16}.
\begin{thm} \label{ultraproduct-crossed}
Let $M$ be any full $\sigma$-finite factor. Let $\sigma : G \rightarrow \Aut(M)$ be an outer action of a discrete group $G$ such that the image of $\sigma(G)$ in $\Out(M)$ is discrete. Then for any cofinal ultrafilter $\omega$ on any directed set $I$, we have $M' \cap (M \rtimes_\sigma G)^{\omega}=\C1$.
\end{thm}
\begin{proof}
When $M$ is a semifinite, Theorem \ref{crossed-product-finite} shows that for any $Z \in  M' \cap (M \rtimes_\sigma G)^\omega $ we have $Zp \in \C p$ for all finite projections $p \in M$. Hence $Z \in \C1$. Therefore $M' \cap (M \rtimes_\sigma G)^\omega=\C 1$.

Now assume that $M$ is of type ${\rm III}$. Choose a faithful state $\varphi \in M_\ast$, $\kappa > 0$ and $a_1,\dots,a_k \in M$ which satisfy the conclusion of Theorem \ref{crossed-product}. Assume by contradiction that $ M' \cap (M \rtimes_\sigma G)^\omega \neq \C 1$. Consider the closed positive definite operator 
$$S = \Delta^{1/2}_{\varphi^\omega} |_{\rL^2(M' \cap (M \rtimes_\sigma G)^\omega) \ominus \C \xi_{\varphi^{\omega}}}$$ 
densily defined on the nonzero Hilbert space $\rL^2(M' \cap (M \rtimes_\sigma G)^\omega) \ominus \C \xi_{\varphi^{\omega}}$. Take some $\lambda > 0$ in the spectrum of $S$. Then, for any $\varepsilon > 0$, we can find $Z \in M' \cap (M \rtimes_\sigma G)^\omega$ such that $\|Z\|_{\varphi^\omega} = 1$, $\varphi^\omega(Z) = 0$ and $\|\xi_{\varphi^\omega} Z - \lambda Z \xi_{\varphi^\omega}\| < \varepsilon$. Hence we can find $z \in M \rtimes_\sigma G$ which satisfies $\|z\|_{\varphi} = 1$, $\varphi(z) = 0$, $\| z \xi_{\varphi} - \lambda^{-1} \xi_{\varphi} z\| \leq 2 \lambda^{-1} \varepsilon$ and $\|z a_j - a_j z\|_{\varphi} \leq \varepsilon$ for every $j \in \{1, \dots, k\}$. This however contradicts Theorem \ref{crossed-product} if $\varepsilon > 0$ is small enough. Therefore, we must have $M' \cap (M \rtimes_\sigma G)^\omega = \C 1$.
\end{proof}

Note that Theorem \ref{ultraproduct-crossed}, applied to $G=\{1\}$ the trivial group, gives the following corollary, which was previously only known in the separable case \cite{AH12} with a very different proof.

\begin{cor} \label{central-sequence}
Let $M$ be any $\sigma$-finite full factor. Then we have $M' \cap M^{\omega}=\C 1$ for any cofinal ultrafilter $\omega$ on any directed set $I$.
\end{cor}

\section{Proofs of Theorem \ref{main-thm-full-tensor-products} and Corollary \ref{cor-full-tensor-products}.}

In this section, we show how to pass from a spectral gap in a full factor $M$ to a relative spectral gap in a tensor product $M \ovt N$, where $N$ is any von Neumann algebra. For this, we will need the following technical lemma which explains why the strengthening of the spectral gap property that we obtained in the previous section is crucial.

\begin{lem}\label{lem-technical}
Let $H$ be any Hilbert space and $S : \mathcal D(S) \to H$ any closed positive densily defined operator. Let $A \in \mathbf B(H)$ be any selfadjoint bounded operator which satisfies
$$\forall \xi \in \mathcal{D}(S), \quad \langle A \xi, \xi \rangle \leq \inf_{\lambda \in \R_+} \| (\lambda S-1) \xi \|^2.$$
Then for every Hilbert space $K$ and every closed positive densily defined operator $T : \mathcal D(T) \to K$, we have
$$\forall \eta \in \mathcal D(S \otimes T), \quad \langle (A \otimes 1) \eta, \eta \rangle \leq  \| (S \otimes T - 1) \eta \|^2.$$
\end{lem}

\begin{proof}
Write $A_0 = \{T^{{\rm i}t} : t \in \R\}$ for the von Neumann subalgebra of $\mathbf B(K)$ generated by $T$. Choose an intermediate maximal abelian subalgebra $A_0 \subset A \subset \mathbf B(K)$. Then we can identify $\left( A \subset \mathbf B(K) \right) = \left( \rL^\infty(X, \mu) \subset \mathbf B(\rL^2(X, \mu)) \right)$ for some measure space $(X, \mu)$. Since $T$ is affiliated with $A = \rL^\infty(X, \mu)$, the spectral theorem for unbounded operators (see e.g.\ \cite[Theorem 5.6.4]{KR97}) shows that we can assume that $T = M_f$ acts by multiplication on $\rL^2(X, \mu)$ by some measurable function $f: (X,\mu) \rightarrow \R_+$.
Then
$$\mathcal{D}(T)=\left\{ g \in \rL^2(X,\mu) : \int_X |fg|^2 \, {\rm d}\mu < +\infty \right\}.$$
Let $\eta=\sum_{i = 1}^p  \xi_i \otimes g_i$ be any element in the algebraic tensor product $\mathcal{D}(S) \odot \mathcal{D}(T)$. Identify $H \otimes K = \rL^2(X, \mu, H)$ and regard $\eta : X \to H : x \mapsto \eta_x=\sum_{i = 1}^p g_i(x)\xi_i$. Observe that $\eta_x \in \mathcal D(S)$ for $\mu$-almost every $x \in X$ and that the map $X \to \R_+ : x \mapsto \| ( f(x)S-1)\eta_x\|^2$ is measurable. Then a simple computation shows that
$$\langle (A \otimes 1) \eta, \eta \rangle = \int_{x \in X} \langle A \eta_x,\eta_x \rangle \, {\rm d}\mu(x)$$
and 
$$ \| (S \odot T - 1)\eta \|^2= \int_{x \in X} \| (f(x)S-1)\eta_x\|^2 \, {\rm d}\mu(x).$$
Hence, we have
$$\forall \eta \in \mathcal{D}(S) \odot \mathcal{D}(T), \quad \langle (A \otimes 1) \eta, \eta \rangle \leq  \| (S \odot T -1) \eta \|^2.$$

Since $S \otimes T$ is the closure of $S \odot T$, for every $\eta \in \mathcal{D}(S \otimes T)$, we can find a sequence $\eta_n \in \mathcal{D}(S) \odot \mathcal{D}(T)$ such that $\eta_n \rightarrow \eta$ and $(S \odot T)\eta_n \rightarrow (S \otimes T)\eta$ as $n \to \infty$. For every $n \in \N$, we have
$$ \langle (A \otimes 1) \eta_n, \eta_n \rangle \leq \| (S \odot T -1) \eta_n \|^2.$$
Therefore we obtain
$$\langle (A \otimes 1) \eta, \eta \rangle \leq \| (S \otimes T -1) \eta \|^2$$
as we wanted.
\end{proof}

\begin{thm}\label{cor-strong-gap}
Let $M$ be any full factor. Choose a state $\varphi \in M_\ast$, a family $a_1,\dots,a_k \in M$ and $\kappa > 0$ which satisfy the conclusion of Theorem \ref{strong-gap}. Let $N$ be any von Neumann algebra  and $\psi \in N_\ast$ be any state. Denote by $\rE_\varphi=\varphi \otimes \id_N : M \ovt N \to N$ the normal conditional expectation induced by $\varphi$. 

Then for all $z \in M \ovt N$ and all $\beta \in \Aut(N)$ we have
\begin{align}\label{eq-cor-spectral-gap}
\|z - \rE_\varphi(z)\|_{\varphi \otimes \psi}^2 \leq \kappa \left( \sum_{j = 1}^k \| za_j -a_j z \|_{\varphi \otimes \psi}^2 + \inf_{\lambda \in \R_+}\|z(\xi_{\varphi} \otimes {\xi_\psi}) - \lambda (\xi_{\varphi} \otimes \beta(\xi_{\psi})) z \|^2 \right)
\end{align}
and for all $z \in M \ovt N$, all $\alpha \in \Aut(M) \setminus \pi_M^{-1}(\mathcal V)$ and all $\beta \in \Aut(N)$ we have
\begin{align}\label{eq-cor-spectral-gap-out}
\| z \|_{\varphi \otimes \psi}^2 \leq \kappa \left( \sum_{j = 1}^k \| z a_j - \alpha(a_j) z \|_{\varphi \otimes \psi}^2 + \inf_{\lambda \in \R_+} \left\|z (\xi_\varphi \otimes \xi_\psi) - \lambda(\alpha(\xi_\varphi) \otimes \beta(\xi_\psi))z \right\|^{2} \right).
\end{align}
\end{thm}

\begin{proof}
{ \bf Proof of the first item \eqref{eq-cor-spectral-gap}. } Let $S=\Delta_\varphi^{1/2}$ and let $p$ be the support of $\varphi$. Then \eqref{eq-spectral-gap-bis} implies that 
$$\forall \xi \in \mathcal D(S), \quad \langle A \xi, \xi \rangle \leq \inf_{\lambda \in \R_+} \| (\lambda S-1)\xi \|^2,$$
where
$$ A=\frac{1}{\kappa}(JpJ - P_{\C \xi_\varphi})-\sum_{j = 1}^k |a_j-Ja_jJ|^{2}.$$
Indeed, it is enough to check this inequality for $\xi$ of the form $x \xi_\varphi$ with $x \in Mp$.

Now, fix $\beta \in \Aut(N)$ any automorphism of $N$ and let $T=\Delta^{1/2}_{\beta(\psi),\psi}$. Now, for every $\lambda \in \R_+$, we apply Lemma \ref{lem-technical} to $S$ and $\lambda T$ and we obtain
$$\forall \eta \in \mathcal D(S \otimes T), \quad \langle (A \otimes 1) \eta, \eta \rangle \leq \inf_{\lambda \in \R_+} \| (\lambda S \otimes T -1)\eta\|^2.$$
Since $S \otimes T=\Delta_{\varphi \otimes \beta(\psi),\varphi \otimes \psi}^{1/2}$  and since the orthogonal projection $P_{\C \xi_\varphi} \otimes 1 : \rL^2(M \ovt N) \to \C \xi_\varphi \otimes \rL^2(N)$ satisfies $P_{\C \xi_\varphi}(z (\xi_{\varphi} \otimes \xi_{\psi})) = \xi_\varphi \otimes \rE_\varphi(z)\xi_\psi$ for all $z \in M \ovt N$, we finally obtain \eqref{eq-cor-spectral-gap}.

{ \bf Proof of the second item \eqref{eq-cor-spectral-gap-out}. } Fix $\alpha \in \Aut(M) \setminus \pi_M^{-1}(\mathcal V)$. Let $S=\Delta^{1/2}_{\alpha(\varphi),\varphi}$. Then \eqref{eq-spectral-gap-out-bis} implies that 
$$\forall \xi \in \mathcal D(S), \quad \langle A \xi, \xi \rangle \leq \inf_{\lambda \in \R_+} \| (\lambda S-1)\xi \|^2,$$
where
$$ A=\frac{1}{\kappa}JpJ - \sum_{j = 1}^k |Ja_jJ-\alpha(a_j)|^{2}.$$
Fix $\beta \in \Aut(N)$ any automorphism. Let $T=\Delta^{1/2}_{\beta(\psi),\psi}$. We now apply Lemma \ref{lem-technical} and we obtain
$$\forall \eta \in \mathcal D(S \otimes T), \quad \langle (A \otimes 1) \eta, \eta \rangle \leq \inf_{\lambda \in \R_+} \| (\lambda S \otimes T -1)\eta\|^2.$$
Since $S \otimes T=\Delta_{\alpha(\varphi) \otimes \beta(\psi),\varphi \otimes \psi}^{1/2}$ we finally obtain \eqref{eq-cor-spectral-gap-out}.
\end{proof}

We also note that we can obtain a much simpler statement in the semifinite case.

\begin{thm}\label{cor-gap-finite}
Let $M$ be any full semifinite factor. Let $\mathcal{V}$ be any neighborhood of $1$ in $\Out(M)$. Let $p \in M$ be any nonzero finite projection. Take $\varphi$, $\kappa > 0$ and $a_1,\dots,a_k$ as in Theorem \ref{spectral-gap-finite}. Let $N$ be any von Neumann algebra  and $\psi \in N_\ast$ be any state. Denote by $\rE_\varphi=\varphi \otimes \id_N : M \ovt N \to N$ the normal conditional expectation induced by $\varphi$. Then for all $z \in M \ovt N$ we have
\begin{align}\label{eq-cor-gap-finite}
\|z - \rE_\varphi(z)\|_{\varphi \otimes \psi}^2 \leq \kappa \sum_{j = 1}^k \| za_j-a_j z \|_{\varphi \otimes \psi}^2
\end{align}
and for all $z \in M \ovt N$ and all $\alpha \in \Aut(M) \setminus \pi_M^{-1}(\mathcal V)$ we have
\begin{align}\label{eq-cor-gap-out-finite}
\| z \|_{\varphi \otimes \psi}^2 \leq \kappa \sum_{j = 1}^k \| z a_j - \alpha(a_j)z \|_{\varphi \otimes \psi}^2.
\end{align}
\end{thm}
\begin{proof}
Let $A=\sum_{j=1}^{k} |a_j-Ja_jJ|^{2} \in \mathbf B(\rL^{2}(M))$. Then, by \eqref{eq-spectral-gap-finite} we have $JpJ-P_{\C \xi_\varphi } \leq \kappa A$. Hence $(JpJ-P_{\C \xi_\varphi })\otimes 1 \leq \kappa (A\otimes 1)$ in $\mathbf B(\rL^{2}(M \ovt N))$ and this gives \eqref{eq-cor-gap-finite}.

Similarly, fix $ \alpha \in \Aut(M) \setminus \pi_M^{-1}(\mathcal V)$ and let $A=\sum_{j=1}^{k} |Ja_jJ-\alpha(a_j)|^{2} \in \mathbf B(\rL^{2}(M))$. Then by \eqref{eq-spectral-gap-out-finite} we have $JpJ \leq \kappa A$. Hence $JpJ \otimes 1 \leq \kappa (A \otimes 1)$ and this gives \eqref{eq-cor-gap-out-finite}.
\end{proof}

\subsection*{Proofs of Theorem \ref{main-thm-full-tensor-products} and Corollary \ref{cor-full-tensor-products}}

\begin{proof}[Proof of Theorem \ref{main-thm-full-tensor-products}]
{\bf Proof of the first statement.} Let $(x_i)_{i \in I}$ be a centralizing net in $\Ball(M\ovt N)$. Let $\varphi \in M_*$ be a state which satisfies the conclusion of Theorem \ref{strong-gap}. Let $y_i=\rE_\varphi(x_i)$ for every $i \in I$. Observe that $(y_i)_{i \in I}$ is a centralizing net in $N$. Moreover, \eqref{eq-cor-spectral-gap} shows that $\lim_i \| x_i-y_i \|_{\varphi \otimes \psi} =0$ for any state $\psi \in N_*$. This means that $(x_i-y_i)p \to 0$ strongly as $i \to \infty$ where $p \in M$ is the support of $\varphi$. Since $(x_i)_{i \in I}$ is asymptotically central, this shows that $(x_i-y_i)q \to 0$ strongly as $i \to \infty$ for any projection $q \in M$ which is equivalent to $p$. Therefore, since $M$ is a factor, we finally obtain $x_i-y_i \to 0$ strongly as $i \to \infty$.

{\bf Proof of the second statement.} Since the natural inclusion $\iota : \Out(M) \times \Out(N) \to \Out(M \ovt N)$ is continuous, we only have to prove that $\iota^{-1}$ is continuous on $\iota(\Out(M) \times \Out(N))$. For this, we have to show that for any net $(\alpha_i \otimes \beta_i)_{i \in I}$ in $\Aut(M) \times \Aut(N)$ such that $\pi_{M \ovt N}(\alpha_i \otimes \beta_i) \to 1$ in $\Out(M \ovt N)$ as $i \to \infty$ we have $\pi_M(\alpha_i) \to 1$ in $\Out(M)$ and $\pi_N(\beta_i) \to 1$ in $\Out(N)$ as $i \to \infty$. We can always assume that $I$ is large enough so that there exists a decreasing net $(\mathcal{W}_i)_{i \in I}$ of $*$-strong neighborhoods of $0$ in $N$ which is cofinal, meaning that for every $*$-strong neighborhood $\mathcal{W}$ of $0$ in $N$ there exists $i \in I$ such that $\mathcal{W}_i \subset \mathcal{W}$.

First, we prove that $\pi_M(\alpha_i) \to 1$ in $\Out(M)$ as $i \to \infty$. Assume by contradiction that this is not the case. Then there exist an open neighborhood $\mathcal V$ of $1$ in $\Out(M)$ and a subnet $(\alpha_j)_{j \in J}$ such that $\pi_M(\alpha_j) \notin \mathcal V$ for every $j \in J$. Choose $\varphi \in M_\ast$, $\kappa > 0$ and $a_1,\dots,a_k$ which satisfy the conclusion of Theorem \ref{strong-gap}. Let $\psi \in N_\ast$ be any state. Since $\pi_{M \ovt N}(\alpha_j \otimes \beta_j) \to 1$ in $\Out(M \ovt N)$ as $j \to \infty$, there exists a net $(u_j)_{j \in J}$ in $\mathcal U(M \ovt N)$ such that $\Ad(u_j)^{-1} \circ (\alpha_j \otimes \beta_j) \to \id_{M \ovt N}$ in $\Aut(M \ovt N)$ as $j \to \infty$. In particular, we have $\lim_j \|u_j a - \alpha_j(a) u_j\|_{\varphi \otimes \psi} = 0$ for every $a \in M$ and $\lim_{j} \|u_j (\xi_\varphi \otimes \xi_\psi) - (\alpha_j(\xi_\varphi) \otimes \beta_j(\xi_\psi))u_j\| = 0$. Since $\|u_j\|_{\varphi \otimes \psi} = 1$ for every $j \in J$, this contradicts \eqref{eq-cor-spectral-gap-out}.

Secondly, we prove that $\pi_N(\beta_i) \to 1$ in $\Out(N)$ as $i \to \infty$. Since $\pi_M(\alpha_i) \to 1$ in $\Out(M)$ we have that $\pi_{M \ovt N}(\alpha_i \otimes \id_N) \to 1$ in $\Out(M \ovt N)$ as $i \to \infty$. Since by assumption, $\pi_{M \ovt N}(\alpha_i \otimes \beta_i) \to 1$ in $\Out(M \ovt N)$, we also have that $\pi_{M \ovt N}(\id_M \otimes \beta_i) \to 1$ in $\Out(M \ovt N)$ as $i \to \infty$. Hence there exists a net $(v_i)_{i\in I}$ in $\mathcal U(M \ovt N)$ such that $\Ad(v_i)^{-1} \circ (\id_M \otimes \beta_i) \to \id_{M \ovt N}$ in $\Aut(M \ovt N)$ as $i \to \infty$. Write $y_i = \rE_\varphi(v_i) \in \Ball(N)$ for every $i \in I$. By using \cite[Lemma 5.2]{Ma16}, we can find for every $i \in I$ an invertible element $x_i \in \Ball(N)$ such that $y_i - x_i \in \mathcal{W}_i$ and $\beta_i^{-1}(y_i) - \beta_i^{-1}(x_i) \in \mathcal{W}_i$. Then we have $x_i-y_i \to 0$ and $\beta_i^{-1}(y_i) - \beta_i^{-1}(x_i) \to 0$ in the $*$-strong topology as $i \to \infty$. Let $x_i=u_i|x_i|$ be the polar decomposition of $x_i$ with $u_i \in \mathcal{U}(N)$. We will show that $\Ad(u_i)^{-1} \circ \beta_i \to \id_N$ in $\Aut(N)$ as $i \to \infty$. For this, it is enough to show that $\lim_{i} \| u_i \xi_\psi - \beta_i(\xi_\psi) u_i\| = 0$ for any state $\psi \in N_*$.

First, since $\Ad(v_i)^{-1} \circ (\id_M \otimes \beta_i) \to \id_{M \ovt N}$ in $\Aut(M \ovt N)$ as $i \to \infty$, we have $\lim_i \| v_ia  - a v_i\|_{\varphi \otimes \psi} = 0$ for every $a \in M$ and $\lim_{i} \| v_i (\xi_\varphi \otimes \xi_\psi) - (\xi_\varphi \otimes \beta_i(\xi_\psi))v_i \| = 0$. Then \eqref{eq-cor-spectral-gap} implies that $\lim_i \|v_i - y_i\|_{\varphi \otimes \psi} = 0$ which implies, since $v_i$ is a unitary, that $\lim_i \| x_i \|_\psi = \lim_i \| y_i \|_\psi=1$. Since $\lim_{i} \| v_i (\xi_\varphi \otimes \xi_\psi) - (\xi_\varphi \otimes \beta_i(\xi_\psi))v_i \| = 0$, using the orthogonal projection $P_{\C \xi_\varphi} \otimes 1 : \rL^2(M \ovt N) \to \C \xi_\varphi \otimes \rL^2(N)$ which satisfies $(P_{\C \xi_\varphi} \otimes 1)(z (\xi_{\varphi} \otimes \xi_{\psi})) = \xi_\varphi \otimes \rE_\varphi(z)\xi_\psi$ and $(P_{\C \xi_\varphi} \otimes 1)((\xi_{\varphi} \otimes \xi_{\psi})z) = \xi_\varphi \otimes \xi_\psi \rE_\varphi(z)$ for every $z \in M \ovt N$, we also have $\lim_{i} \| y_i \xi_\psi - \beta_i(\xi_\psi) y_i\| = 0$ and thus $\lim_{i} \| x_i \xi_\psi - \beta_i(\xi_\psi) x_i\| = 0$. In particular, we obtain $\lim_i \| \beta_i^{-1}(x_i^*) \|_\psi=\lim_i \| x_i \|_\psi=1$. Since $x_i \in \Ball(N)$, this easily implies that $\lim_i \| x_i-u_i\|_\psi=\lim_i \| \beta_i^{-1}(x_i^*-u_i^*) \|_\psi=0$. This, combined with $\lim_{i} \| x_i \xi_\psi - \beta_i(\xi_\psi) x_i\| = 0$, gives $\lim_{i} \| u_i \xi_\psi - \beta_i(\xi_\psi) u_i\| = 0$ as we wanted. This finally shows that $\pi_N(\beta_i) \to 1$ in $\Out(N)$ as $i \to \infty$.
\end{proof}

\begin{proof}[Proof of Corollary \ref{cor-full-tensor-products}] 
{\bf Proof of the first statement.} Let $(x_i)_{i \in I}$ be a uniformly bounded centralizing net in $M_1 \ovt M_2$. Since $M_1$ is full, then by Theorem \ref{main-thm-full-tensor-products} we can find a uniformly bounded centralizing net $(y_i)_{ i \in I}$ in $M_2$ such that $x_i-y_i \to 0$ strongly as $i \to \infty$. But $M_2$ is also full hence we can find a bounded net $(\lambda_i)_{i \in I}$ in $\C$ such that $y_i-\lambda_i 1 \to 0$ strongly as $i \to \infty$. Therefore we get $x_i- \lambda_i 1 \to 0$ strongly as $i \to \infty$. This means that $M_1 \ovt M_2$ is full.

{\bf Proof of the second statement.} Choose two faithful semifinite normal weights $\varphi_1$ and $\varphi_2$ on $M_1$ and $M_2$ respectively. Put $M=M_1 \ovt M_2$ and $\varphi = \varphi_1 \otimes \varphi_2$ so that $\sigma_t^\varphi = \sigma_t^{\varphi_1} \otimes\sigma_t^{\varphi_2}$ for every $t \in \R$. By Theorem \ref{main-thm-full-tensor-products}, the natural homomorphism $\Out(M_1) \times \Out(M_2) \to \Out(M)$ is a homeomorphism onto its range. Therefore, for every net $(t_i)_{i \in I}$ in $\R$, we have
\begin{align*}
t_i \to 0 \text{ w.r.t. } \tau(M_2) \quad & \Leftrightarrow \quad \pi_{M}(\sigma_{t_i}^{\varphi_1} \otimes \sigma_{t_i}^{\varphi_2}) \to 1 \text{ in } \Out(M) \\ 
& \Leftrightarrow \quad \left(\pi_{M_1}(\sigma_{t_i}^{\varphi_1}),\pi_{M_2}(\sigma_{t_i}^{\varphi_2}) \right) \to 1 \text{ in } \Out(M_1) \times \Out(M_2) \\
& \Leftrightarrow \quad t_i \to 0 \text{ w.r.t. } \tau(M_1) \text{ and } \tau(M_2). \qedhere
\end{align*}
\end{proof}

We can prove a more precise statement by using Ocneanu's ultraproduct. 

\begin{thm} \label{ultraproduct_full}
Let $M$ be any full factor and $N$ any von Neumann algebra. Assume that $M$ and $N$ are $\sigma$-finite. Then for any cofinal ultrafilter $\omega$ on any directed set $I$, we have $M' \cap (M \ovt N)^{\omega}=N^{\omega}$.
\end{thm}
\begin{proof}
When $M$ is a semifinite, Theorem \ref{cor-gap-finite} shows that for any $Z \in  M' \cap (M \ovt N)^\omega$ we have $Zp \in N^{\omega}p$ for all finite projections $p \in M$. Hence $Z \in N^{\omega}$. Therefore $M' \cap (M \ovt N)^\omega=N^{\omega}$.

Now assume that $M$ is of type ${\rm III}$. Choose a faithful state $\varphi \in M_\ast$ and $a_1,\dots,a_k \in M$ which satisfy the conclusion of Theorem \ref{strong-gap}. Let $\psi \in N_\ast$ be any faithful state. Denote by $\rE_N : M \ovt N \to N$ the unique $(\varphi \otimes \psi)$-preserving conditional expectation. Assume by contradiction that $N^\omega \subsetneq M' \cap (M \ovt N)^\omega$. Since the inclusion $N^\omega \subset M' \cap (M \ovt N)^\omega$ is globally invariant under the modular automorphism group $\sigma^{(\varphi \otimes \psi)^\omega}$, we may regard the operator 
$$S = \Delta^{1/2}_{(\varphi \otimes \psi)^\omega} |_{\rL^2(M' \cap (M \ovt N)^\omega) \ominus \rL^2(N^\omega)}$$ 
as a closed positive definite densily defined operator on the nonzero Hilbert space $\rL^2(M' \cap (M \ovt N)^\omega) \ominus \rL^2(N^\omega)$. Take some $\lambda > 0$ in the spectrum of $S$. Then, for any $\varepsilon > 0$, we can find $Z \in M' \cap (M \ovt N)^\omega$ such that $\|Z\|_{(\varphi \otimes \psi)^\omega} = 1$, $\rE_{N^\omega}(Z) = 0$ and $\|\xi_{(\varphi \otimes \psi)^\omega} Z - \lambda Z \xi_{(\varphi \otimes \psi)^\omega}\| < \varepsilon$. Hence we can find $z \in M \ovt N$ which satisfies $\|z\|_{\varphi \otimes \psi} = 1$, $\rE_{N}(z) = 0$, $\| z \xi_{\varphi \otimes \psi} - \lambda^{-1} \xi_{\varphi \otimes \psi} z\| \leq 2 \lambda^{-1} \varepsilon$ and $\|z a_j - a_j z\|_{\varphi \otimes \psi} \leq \varepsilon$ for every $j \in \{1, \dots, k\}$. This however contradicts \eqref{eq-cor-spectral-gap} if $\varepsilon > 0$ is small enough. Therefore, we have $M' \cap (M \ovt N)^\omega = N^\omega$.
\end{proof}

We finish this section by noting some important consequences of Conjecture \ref{conj-gap}.

\begin{prop}\label{prop-conjecture}
Let $M$ be any $\sigma$-finite full factor of type $\III$. Suppose that there exist a faithful state $\varphi \in M_\ast$, $\kappa > 0$ and a family $a_1,\dots,a_k \in M$ satisfying $a_j \varphi a_j^* \leq \varphi$ for every $j \in \{ 1, \dots, k \}$ such that we have
\begin{equation*}
\forall x \in M, \quad \| x -\varphi(x) 1 \|_\varphi^2 \leq \kappa \sum_{j = 1}^k \| x a_j - a_j x \|_\varphi^2.
\end{equation*}
Then we have $\mathbf K(\rL^2(M)) \subset \rC^*(M, JMJ)$. 

Moreover, every uniformly bounded net $(x_i)_{i \in I}$ in $M$ which is asymptotically central, i.e.\ $x_ia-ax_i \to 0$ strongly as $i \to \infty$ for all $a \in M$, must be trivial, i.e.\ there exists a bounded net $(\lambda_i)_{i \in I}$ in $\C$ such that $x_i-\lambda_i 1 \to 0$ strongly as $i \to \infty$.
\end{prop}

Note that the second part of Proposition \ref{prop-conjecture} does not follow from Corollary \ref{central-sequence}. Indeed, a uniformly bounded central net may not represent an element of the Ocneanu ultraproduct.

\begin{proof}
By assumption, we know that for every $j \in \{1, \dots, k\}$, there exists $b_j \in \Ball(M)$ such that $a_j \xi_\varphi = \xi_\varphi b_j^* = Jb_jJ \xi_\varphi$ (see Remark \ref{rem-standard-form}). Put $T = \sum_{j = 1}^k |a_j - Jb_j J|^2 \in \rC^*(M, JMJ)$ and observe that $T \xi_\varphi = 0$. Note that $\rL^2(M) = \C\xi_\varphi \oplus \overline{\ker(\varphi)\xi_\varphi}$. For every $x \in \ker(\varphi)$, we have
\begin{align*}
\langle T x \xi_\varphi, x\xi_\varphi \rangle &= \sum_{j = 1}^k \|(a_j - Jb_jJ) x \xi_\varphi \|^2 \\
&= \sum_{j = 1}^k \|(x a_j - a_j x) \xi_\varphi\|^2 \\
& \geq \frac1\kappa \|x \xi_\varphi\|^2.
\end{align*}
Therefore, the eigenvalue $0$ has multiplicity $1$ and the spectrum of $T$ is contained in $\{0\} \cup [\frac1\kappa, \|T\|_\infty]$. Thus, by continuous functional calculus, we have that the orthogonal projection $\mathbf 1_{\{0\}}(T) = P_{\C \xi_\varphi} : \rL^2(M) \to \C\xi_\varphi$ belongs to $\rC^*(M, JMJ)$. For all $x, y \in M$, we have that $y P_{\C \xi_\varphi} x^* \in \rC^*(M, JMJ)$ is the rank one operator $\rL^2(M) \to \rL^2(M) : \eta \mapsto \langle \eta, x \xi_\varphi \rangle y \xi_\varphi$. By density of $M\xi_\varphi$ in $\rL^2(M)$, it follows that $\rC^*(M, JMJ)$ contains all rank one operators on $\rL^2(M)$ and hence $\mathbf K(\rL^2(M)) \subset \rC^*(M, JMJ)$. The second part is obvious.
\end{proof}

\section{Factors with almost periodic states: proof of Theorem \ref{main-thm-almost-periodic}}

Let $M$ be any full factor of type $\mathrm{III}$ with separable predual which possesses an almost periodic faithful normal state. Put $\Gamma = \Sd(M) \leq \R^*_+$ (see \cite{Co74}). Regarding $\Gamma$ as a discrete group, denote by $G = \widehat \Gamma$ its Pontryagin dual, which is a compact group. We view $\R$ as a dense subgroup of $G$.

By \cite[Theorem 4.7]{Co74}, there exists a unique (up to scaling and unitary conjugacy) almost periodic faithful normal semifinite weight $\psi$ on $M$ with $\psi(1)=+\infty$ which satisfies one of the following equivalent properties:
\begin{itemize}

\item [$(\rm i)$] The point spectrum of $\Delta_\psi$ is equal to $\Sd(M)$.

\item [$(\rm ii)$] $M_\psi$ is a factor.

\item [$(\rm iii)$] $(M_\psi)' \cap M=\C 1$.

\end{itemize}
Denote by $N = M_\psi$ the centralizer of the weight $\psi$. By \cite[Corollary 4.12]{Co74}, there exists a trace-scaling action $\theta : \Gamma \curvearrowright N$ such that we have the following isomorphism of inclusions
$$\left( M_\psi \subset M \right) \cong \left( N \subset N \rtimes_\theta \Gamma \right).$$
We will use the identification $M = N \rtimes \Gamma$ and regard $N \subset M$. We denote by $(v_\gamma)_{\gamma \in \Gamma}$ the canonical unitaries in $M = N \rtimes_{\theta} \Gamma$ implementing the action $\theta : \Gamma \curvearrowright N$ and by $\rE_N : M \to N$ the canonical faithful normal conditional expectation defined by $\rE_N(x v_\gamma) = \delta_{\gamma, 1_\Gamma} \, x v_\gamma$ for every $x \in N$ and every $\gamma \in \Gamma$. We moreover denote by $\widehat\theta : G \curvearrowright M$ the dual action defined by $\widehat \theta_g(x) = x$ for every $x \in N$ and $\widehat\theta_g(v_\gamma) = \langle g, \gamma\rangle \, v_\gamma$ for every $\gamma \in \Gamma$.

The proof of Theorem \ref{main-thm-almost-periodic} relies on the following key result that will be combined with Theorem \ref{crossed-product-finite}. 

\begin{thm}\label{thm-crossed-product}
Keep the same notation as above. Then $N$ is full and the image of $\theta(\Gamma)$ in $\Out(N)$ is discrete.
\end{thm}

The conclusion of Theorem \ref{thm-crossed-product} may seem surprising at first glance. Indeed, even though $\Gamma$ may be dense in $\R^*_+$, the image of $\theta(\Gamma)$ in $\Out(N)$ is nevertheless always discrete.

\begin{proof}[Proof of Theorem \ref{thm-crossed-product}]
First, we have that $N$ isomorphic to the {\em discrete core} of $M$ and hence $N$ is full by \cite[Proposition 5]{TU14}. 

Next, we prove that the image of $\theta(\Gamma)$ in $\Out(N)$ is discrete. By contradiction, assume that the image of $\theta(\Gamma)$ in $\Out(N)$ is not discrete. Then there exists a sequence $(\gamma_n)_{n \in \N}$ in $\Gamma \setminus \{1_\Gamma\}$ such that $\pi_N(\theta_{\gamma_n}) \to 1$ in $\Out(N)$ as $n \to \infty$. Since $\Out(N) = \Aut(N)/\Inn(N)$, there exists a sequence of unitaries $(u_n)_{n \in \N}$ in $\mathcal U(N)$ such that $\Ad(u_n) \circ \theta_{\gamma_n} \to \id_N$ in $\Aut(N)$ as $n \to \infty$. Since $\Gamma$ is abelian, for every $\gamma \in \Gamma$, we have 
\begin{align*}
\Ad(\theta_{\gamma}(u_n)) \circ \theta_{\gamma_n} &= \theta_{\gamma} \circ \Ad(u_n) \circ \theta_{\gamma}^{-1} \circ \theta_{\gamma_n}  \\
&= \theta_{\gamma} \circ \Ad(u_n) \circ \theta_{\gamma_n} \circ \theta_{\gamma}^{-1} \to \id_N \quad \text{in } \Aut(N) \text{ as } n \to \infty.
\end{align*}
Since the map $\Aut(M) \times \Aut(M) \to \Aut(M) : (g, h) \mapsto gh^{-1}$ is continuous for the $u$-topology, this further implies that for every $\gamma \in \Gamma$, we have
\begin{align*}
\Ad(u_n\theta_{\gamma}(u_n^{*})) &= \Ad(u_n) \circ \Ad(\theta_{\gamma}(u_n))^{-1} \\
&= \left( \Ad(u_n) \circ \theta_{\gamma_n} \right) \circ \left(\Ad(\theta_{\gamma}(u_n)) \circ \theta_{\gamma_n} \right)^{-1} \to \id_N \quad \text{ in } \Aut(N) \text{ as } n \to \infty.
\end{align*} 

Fix a nonprincipal ultrafilter $\omega \in \beta(\N) \setminus \N$. Since $N$ is full and hence $N_\omega = \C 1$, there exists a mapping $g : \Gamma \to \mathbf T$ such that for every $\gamma \in \Gamma$, we have that $u_n \theta_{\gamma}(u_n^{*}) - g(\gamma)1 \to 0$ strongly as $n \to \omega$.  It is easy to see that $g : \Gamma \to \mathbf T$ is a group homomorphism which means exactly that $g \in G= \widehat{\Gamma}$. 

\begin{claim}\label{claim}
We have $g = 1_G$ and $\Ad(u_n v_{\gamma_n}) \to \id_M$ in $\Aut(M)$ as $n \to \omega$. 
\end{claim}

\begin{proof}[Proof of the Claim]
Firstly, we have that $\Ad(u_n)(v_\gamma) = u_n v_\gamma u_n^{*}\to \langle g, \gamma \rangle \, v_\gamma = \widehat\theta_g(v_\gamma)$ strongly as $n \to \omega$ for every $\gamma \in \Gamma$. Since $\Gamma$ is abelian, this further implies that  $\Ad(u_n v_{\gamma_n})(v_\gamma) \to \widehat\theta_g(v_\gamma)$ strongly as $n \to \omega$ for every $\gamma \in \Gamma$. Since $\Ad(u_n) \circ \theta_{\gamma_n} \to \id_N$ in $\Aut(N)$ as $n \to \infty$, we have that $\Ad( u_n v_{\gamma_n} )(x) \to x$ strongly as $n \to \infty$ for every $x \in N$. This altogether implies that $((\widehat\theta_g)^{-1} \circ \Ad(u_n v_{\gamma_n}) )(x) - x \to 0$ strongly as $n \to \infty$ for every $x \in N \rtimes_\theta^{\alg} \Gamma$. Note that $N \rtimes_\theta^{\alg} \Gamma$ is strongly dense in $M = N \rtimes \Gamma$.

Secondly, fix a faithful normal state $\varphi \in M_\ast$ such that $\varphi = \varphi \circ \rE_N$. Then we have $\varphi = \varphi |_N \circ \rE_N$. We also have $\varphi \circ \widehat \theta_g = \varphi$. Since $\Ad(u_n) \circ \theta_{\gamma_n} \to \id_N$ in $\Aut(N)$ as $n \to \infty$, we have $\lim_n \|\varphi |_N \circ \Ad(u_n) \circ \theta_{\gamma_n} - \varphi |_N\| = 0$. We moreover have 
\begin{align*}
 & \| \varphi \circ \Ad( u_n v_{\gamma_n} ) - \varphi \| \\
 &=  \| \varphi|_N \circ \rE_N \circ \Ad( u_n v_{\gamma_n} ) - \varphi |_N \circ \rE_N \| \\
&= \| \varphi|_N  \circ \Ad(u_n) \circ \rE_N \circ \Ad( v_{\gamma_n} ) - \varphi|_N \circ \rE_N \| \quad (\text{since } u_n \in \mathcal U(N))\\
&= \| \varphi|_N  \circ \Ad(u_n) \circ \theta_{\gamma_n} \circ \rE_N - \varphi|_N \circ \rE_N \| \quad (\text{since } \theta_{\gamma_n} \circ \rE_N = \rE_N \circ \Ad(v_{\gamma_n}))\\
&= \| \varphi|_N \circ \Ad(u_n) \circ \theta_{\gamma_n}- \varphi|_N \| \to 0 \quad \text{as } n \to \infty.
\end{align*}
This implies that $\lim_n \|\varphi \circ (\widehat\theta_g)^{-1} \circ \Ad(u_n v_{\gamma_n})  - \varphi \| = \lim_n \|\varphi \circ \Ad(u_n v_{\gamma_n})  - \varphi \| = 0$. 

By Remark \ref{rem-convergence}, we obtain $(\widehat\theta_g)^{-1} \circ \Ad(u_n v_{\gamma_n}) \to \id_M$ in $\Aut(M)$ as $n \to \omega$. Thus, $\Ad(u_n v_{\gamma_n}) \to \widehat\theta_g$ in $\Aut(M)$ as $n \to \omega$. Since $M$ is full, \cite[Theorem 3.1]{Co74} implies that there exists $u \in \mathcal U(M)$ such that $\widehat\theta_g = \Ad(u)$. For every $x \in N$, since $\widehat\theta_g(x) = x$, we have $u x u^* = x$ and hence $u \in N' \cap M = \C 1$. Thus, $u \in \mathbf T 1$ and hence $\widehat\theta_g = \id_M$ which implies that $g = 1_G$. We finally have that $\Ad(u_n v_{\gamma_n}) \to \id_M$ in $\Aut(M)$ as $n \to \omega$. The finishes the proof of the Claim.
\end{proof}  

Since $M$ is full and since $U = (u_n v_{\gamma_n})^\omega \in M_\omega = \C 1$ by Claim \ref{claim}, we obtain $\lim_{n \to \omega} \|u_n v_{\gamma_n} - \psi(u_n v_{\gamma_n})1\|_\psi = \|U - \psi^\omega(U)1\|_{\psi^\omega} = 0$. Since $\psi(u_n v_{\gamma_n}) = (\psi \circ \rE_N)(u_n v_{\gamma_n}) = \psi(u_n \rE_N(v_{\gamma_n})) = \delta_{\gamma_n, 1_\Gamma} \, \psi(u_n)$ and since $u_n v_{\gamma_n} \in \mathcal U(M)$ for every $n \in \N$, we have $\{n \in \N : \gamma_n = 1_\Gamma\} \in \omega$ and hence $\{n \in \N : \gamma_n = 1_\Gamma\} \neq \emptyset$. This contradicts the assumption that $\gamma_n \in \Gamma \setminus \{1_\Gamma\}$ for every $n \in \N$. Therefore, the image of $\theta(\Gamma)$ in $\Out(N)$ is discrete.
\end{proof}

\begin{proof}[Proof of Theorem \ref{main-thm-almost-periodic}]
Put $\mathcal{M}=M \ovt F$ where $F$ is a type $\mathrm{I}_\infty$ factor with separable predual. Define the almost periodic faithful normal semifinite weight on $\mathcal{M}$ by the formula $\psi = \varphi \otimes \mathrm{Tr}_F$.  Put $p=1 \otimes e$ with $e \in F$ any minimal projection. Then we have $(M, \varphi) \cong (p\mathcal{M}p, p \psi p)$. As we explained before, we have a crossed product decomposition $\mathcal{M}\cong \mathcal{M}_\psi \rtimes_\theta \Gamma$ and we moreover know that $\mathcal{M}_\psi$ is full and that the image of $\theta(\Gamma)$ in $\Out(\mathcal{M}_\psi)$ is discrete by Theorem \ref{thm-crossed-product}. Hence we may apply Theorem \ref{crossed-product-finite} to $\mathcal{M}_\psi$ and we obtain $\kappa > 0$ and $a_1,\dots,a_k \in p\mathcal{M}_\psi p=M_\varphi$ such that for all $x \in \mathcal{M}$, and in particular for all $x \in M=p\mathcal{M}p$, we have
\begin{equation*}
 \| x- \varphi(x)1 \|_\varphi^{2} \leq \kappa \sum_{j=1}^k \| a_k x-xa_k \|_\varphi^2. \qedhere
\end{equation*}
\end{proof}

\section{Unique McDuff decomposition: proof of Theorem \ref{main-thm-mcduff}}

We will use the following terminology throughout this section. Let $M$ be any $\sigma$-finite von Neumann algebra and $A \subset 1_A M 1_A$ any $B \subset 1_B M 1_B$ any von Neumann subalgebras with expectation. Following \cite{Po01, Po03, HI15}, we say that $A$ {\em embeds with expectation into} $B$ {\em inside} $M$ and write $A \preceq_M B$, if there exist projections $e \in A$ and $f \in B$, a nonzero partial isometry $v \in eMf$ and a unital normal $\ast$-homomorphism $\theta: eAe \to fBf$ such that the inclusion $\theta(eAe) \subset fBf$ is with expectation and $av = v\theta(a)$ for all $a \in eAe$. We refer the reader to \cite{HI15, BH16} for further information regarding Popa's intertwining theory.

We first need to prove the following result that will be used in the proof of Theorem \ref{main-thm-mcduff}. This is nothing but a generalization of \cite[Lemma 4.6]{HI15}.

\begin{lem}\label{lemma intertwining in tensor}
Let $M$ and $N$ be any $\sigma$-finite von Neumann algebras, $1_A$ and $1_B$ any nonzero projections in $M$, $A \subset 1_A M 1_A$  and $B\subset 1_{B}M1_{B}$ any von Neumann subalgebras with expectation. We will simply denote by $B \ovt N$ the von Neumann subalgebra of $(1_B \otimes 1_N)(M \ovt N)(1_B \otimes 1_N)$ generated by $B \otimes \C 1_N$ and $\C1_B \otimes N$.

The following conditions are equivalent:
\begin{itemize}
\item [$(\rm i)$] $A \preceq_M B$.
\item [$(\rm ii)$] $A \otimes \C 1_N \preceq_{M \ovt N} B \otimes \C 1_N$.
\item [$(\rm iii)$] $A \otimes \C 1_N \preceq_{M \ovt N} B \ovt N$.
\end{itemize}
\end{lem}

\begin{proof}
It is obvious that $(\rm i) \Rightarrow (\rm ii) \Rightarrow (\rm iii)$.

 $(\rm iii) \Rightarrow (\rm i)$ Put $\widetilde B = B \oplus \C 1_B^\perp$ and observe that $\widetilde B \subset M$ is with expectation. We have the canonical identifications $\langle M \ovt N, \widetilde B \ovt N \rangle = \langle M, \widetilde B\rangle \ovt N$ with $\rL^2(M \ovt N) = \rL^2(M) \otimes \rL^2(N)$ and $J^{M \ovt N} = J^M \otimes J^N$. Since $A \otimes \C 1_N \preceq_{M \ovt N} B \ovt N$, by \cite[Theorem 2 $(\rm i) \Rightarrow (\rm iii)$]{BH16}, there exists a normal $(A \otimes \C 1_N)$-$(A \otimes \C 1_N)$-bimodular completely positive map $\widetilde \Phi : \langle M \ovt N, \widetilde B \ovt N \rangle \to A \otimes \C 1_N$ such that $\widetilde \Phi((1_A \otimes 1_N) J^{M \ovt N} (1_B \otimes 1_N) J^{M \ovt N}) \neq 0$. Define the normal $A$-$A$-bimodular completely positive map $\Phi : \langle M, \widetilde B\rangle \to A$ by $\widetilde \Phi(T \otimes 1_N) = \Phi(T) \otimes 1_N$ for all $T \in \langle M, \widetilde B\rangle$. Since $\Phi(1_A J^M 1_B J^M) \otimes 1_N = \widetilde \Phi(1_A J^M 1_B J^M \otimes 1_N) = \widetilde \Phi((1_A \otimes 1_N) J^{M \ovt N} (1_B \otimes 1_N) J^{M \ovt N}) \neq 0$, we have $\Phi(1_A J^M 1_B J^M)  \neq 0$. By \cite[Theorem 2 $(\rm iii) \Rightarrow (\rm i)$]{BH16}, we obtain $A \preceq_M B$.
\end{proof}

We next define the notion of {\em stable unitary conjugacy} for tensor product decompositions. 

\begin{df} \label{stable_conj}
Let $M$ be any $\sigma$-finite factor. A {\em tensor product decomposition} of $M$ is a pair $(A_1,A_2)$ of subalgebras of $M$ such that $M=A_1 \ovt A_2$. Two decompositions $(A_1,A_2)$ and $(B_1,B_2)$ are said to be:
\begin{itemize}

\item {\em unitarily conjugate} if there exists $u \in \mathcal U(M)$ such that $uA_i u^*=B_i$ for all $i \in \{1,2 \}$. We then write $(A_1,A_2) \sim (B_1,B_2)$.

\item {\em stably unitarily conjugate} if there exist type $\mathrm{I}$ factors with separable predual $F_1$ and $F_2$ such that the two tensor product decompositions $(A_1 \ovt F_1,A_2 \ovt F_2)$ and $(B_1 \ovt F_1, B_2 \ovt F_2)$ are unitarily conjugate in $M \ovt F_1 \ovt F_2$. We then write $(A_1,A_2) \sim_\infty (B_1,B_2)$.

\end{itemize}
\end{df}

We record in Proposition \ref{prop-useful} below several useful properties of the notion of (stable) unitary conjugacy.

\begin{prop}\label{prop-useful}
Let $M$ be any $\sigma$-finite factor. The following properties hold true:
\begin{itemize}

\item [$(\rm i)$] The relations $\sim$ and $\sim_\infty$ are equivalence relations.

\item [$(\rm ii)$]  If $M=A_1 \ovt A_2$ is a tensor product decomposition of $M$ and if $F$ is a type $\mathrm{I}$ factor with separable predual then $(A_1 \ovt F,A_2) \sim_\infty (A_1,F \ovt A_2)$ in $M \ovt F$.

\item [$(\rm iii)$]  If $M=A \ovt F=B \ovt G$ are two tensor product decompositions of $M$ where $F$ and $G$ are infinite type $\mathrm{I}$ factors with separable predual then we can find a nonzero partial isometry $v \in M$ with $v^*v=p \in A$ and $vv^*=q \in B$ such that $v \, pAp \, v^*=qBq$ and $v \, pFp \, v^*=qGq$. If both $A$ and $B$ are infinite then we can choose $p=q=1$ so that $(A,F) \sim (B,G)$.

\item [$(\rm iv)$]  If $M=A_1 \ovt A_2=B_1 \ovt B_2$ are two tensor product decompositions of $M$ then we have $(A_1,A_2) \sim_\infty (B_1,B_2)$ if and only if there exist nonzero projections $p_i \in A_i$ and $q_i \in B_i$ and a partial isometry $v \in M$ with $v^*v=p:=p_1p_2$ and $vv^*=q:=q_1q_2$ such that $v \, pA_1p \, v^*=qB_1q$ and $v \, pA_2p \, v^*=qB_2q$. If $A_i$ and $B_i$ are infinite for every $i \in \{1, 2\}$ then we can choose $p=q=1$ so that $(A_1,A_2) \sim (B_1,B_2)$.

\item [$(\rm v)$]  If $M=A_1 \ovt A_2=B_1 \ovt B_2$ are two tensor product decompositions of $M$ then we have $A_1 \preceq_M B_1$ if and only if there exists a tensor product decomposition $B_1=C \ovt D$ such that $(A_1,A_2) \sim_\infty (C,D \ovt B_2)$.
\end{itemize}
\end{prop}

\begin{proof}
$(\rm i)$ It is obvious that $\sim$ is an equivalence relation and that $\sim_\infty$ is symmetric and reflexive. Let us prove that $\sim_\infty$ is transitive. Suppose that $(A_1,A_2)$, $(B_1,B_2)$ and $(C_1,C_2)$ are three tensor product decompositions of $M$ such that $(A_1,A_2) \sim_\infty (B_1,B_2)$ and $(B_1,B_2) \sim_\infty (C_1,C_2)$. Then we can find type $\mathrm{I}$ factors with separable predual $F_i$ and $G_i$ for $i \in \{1,2\}$ such that $(A_1 \ovt F_1,A_2 \ovt F_2) \sim (B_1 \ovt F_1, B_2 \ovt F_2)$ in $M \ovt F_1 \ovt F_2$ and $(B_1 \ovt G_1, B_2 \ovt G_2) \sim (C_1 \ovt G_1, C_2 \ovt G_2)$ in $M \ovt G_1 \ovt G_2$. Then we naturally have $(A_1 \ovt F_1 \ovt G_1, A_2 \ovt F_2 \ovt G_2) \sim (C_1 \ovt F_1 \ovt G_1, C_2 \ovt F_2 \ovt G_2)$ in $M \ovt (F_1 \ovt G_1) \ovt (F_2 \ovt G_2)$. Hence $(A_1,A_2) \sim_\infty (C_1,C_2)$ as we wanted.

$(\rm ii)$ Let $G_1,G_2$ be two infinite type $\mathrm{I}$ factors with separable predual. Then $G_1$ is isomorphic to $G_1 \ovt F$ and $F \ovt G_2$ is isomorphic to $G_2$. Hence we can find an automorphism $\phi$ of $G_1 \ovt F \ovt G_2$ such that $\phi(G_1)=G_1 \ovt F$ and $\phi(F \ovt G_2)=G_2$. Since $G_1 \ovt F \ovt G_2$ is a type $\mathrm{I}$ factor, $\phi$ is inner. Hence $(G_1,F \ovt G_2) \sim (G_1 \ovt F, G_2)$ in $G_1 \ovt F \ovt G_2$. Therefore $(A_1 \ovt G_1, A_2 \ovt F \ovt G_2) \sim (A_1 \ovt F \ovt G_1, A_2 \ovt G_2)$ which means exactly that $(A_1,F \ovt A_2) \sim_\infty (A_1 \ovt F, A_2)$ as we wanted.

$(\rm iii)$ Take $(e_{k,l})_{k,l \in \N}$ and $(f_{k,l})_{k,l \in \N}$ two systems of matrix units for $F$ and $G$ respectively. Choose $p$ and $q$ two nonzero projections in $A$ and $B$ respectively such that $p e_{00}$ and $q f_{00}$ are equivalent in $M$. If both $A$ and $B$ are infinite, then we can take $p=q=1$. Take $v \in M$ such that $v^*v=p e_{00}$ and $vv^*=q f_{00}$ and define $V=\sum_{k \in \N}f_{k0}ve_{0k}$. Then we have $V^*V=p$, $VV^*=q$ and $V \, pAp \, V^*=qBq$ and $V \, pFp \,V^*=qGq$.  

$(\rm iv)$  Firstly, we prove the ``if" direction. Take nonzero projections $p_i \in A_i$ and $q_i \in B_i$ and a partial isometry $v \in M$ with $v^*v=p:=p_1p_2$ and $vv^*=q:=q_1q_2$ such that $v \, pA_1p \, v^*=qB_1q$ and $v \, pA_2p \,v^*=qB_2q$. Take $F_1,F_2$ two infinite type $\mathrm{I}$ factors with separable predual. Then there exist isometries $r_i \in A_i \ovt F_i$ and $s_i \in B_i \ovt F_i$ such that $r_ir_i^*=p_i$ and $s_is_i^*=q_i$. Then $U=(s_1s_2)^*v(r_1r_2)$ is a unitary in $M \ovt F_1 \ovt F_2$ such that $U(A_i \ovt F_i)U^*=B_i \ovt F_i$.

Secondly, we prove the ``only if" direction. Assume that $(A_1,A_2) \sim_\infty (B_1,B_2)$. Then we can find two  type $\mathrm{I}$ factors with separable predual $F_1,F_2$ and a unitary $u \in \mathcal U(M \ovt F_1 \ovt F_2)$ such that $u(A_i \ovt F_i)u^*=B_i \ovt F_i$ for every $i\in \{1,2\}$. Now, by applying $(\rm iii)$ to the two tensor product decompositions of $N_i=uA_iu^* \ovt uF_i u^*=B_i \ovt F_i$, we see that we can find a nonzero partial isometry $v_i \in N_i$ with $v_i^*v_i=up_iu^* \in uA_iu^*$ and $v_iv_i^*=q_i \in B_i$ (and $p_i=q_i=1$ if $A_i$ and $B_i$ are infinite) such that $(v_iu) \, p_iA_ip_i \, (v_iu)^*=q_iB_iq_i$ and $(v_iu) \, p_iF_ip_i \, (v_iu)^*=q_iF_iq_i$ for every $i\in \{1,2\}$. Then there is a unique automorphism $\phi_i$ of $F_i$ such that $\phi(x)v_iu=v_iux$ for all $x \in F_i$ and $\phi_i$ must be inner because $F_i$ is of type $\mathrm{I}$. Thus, up to replacing $v_i$ by $w_iv_i$ for some unitary $w_i \in F_i$, we may assume that $v_iu$ commutes with $F_i$.  Finally, if we let $V=v_1 v_2u \in M \ovt F_1 \ovt F_2=N_1 \ovt N_2$, we see that $V$ commutes with $F_1 \ovt F_2$ which means that $V \in M$ and we have $V^*V=p:=p_1p_2$, $VV^*=q:=q_1q_2$ and $V \, pA_ip \,V^*=q_iB_iq$ for every $i\in \{1,2 \}$. Moreover, if all $A_i$ and $B_i$ are infinite then we can choose $p=q=1$ and hence $(A_1,A_2) \sim (B_1,B_2)$.

$(\rm v)$ By \cite[Lemma 4.13]{HI15}, we can find nonzero projections $p_i \in A_i$ and $q_i \in B_i$ and a partial isometry $v \in M$ with $v^*v=p=p_1p_2$ and $vv^*=q=q_1q_2$ such that $v \, p A_1 p \, v^* \subset qB_1q$. By reducing $q_1$ if necessary, we may further assume that $q_1$ is a minimal projection in some type $\mathrm{I}$ subfactor $K$ of $B_1$. By the second part of \cite[Lemma 4.13]{HI15}, we know that we have a decomposition $q_1B_1q_1=P \ovt Q$ such that $v \, pA_1p \, v^*=qPq$ and $v \, pA_2p \, v^*=Q \ovt qB_2q$. Now, pick an isomorphism $\phi : q_1B_1q_1 \ovt K \rightarrow B_1$ such that $\phi(x \otimes q_1)=x$ for all $x \in q_1B_1q_1$ and let $C=\phi(P \ovt K)$ and $D=\phi(Q)$. Then we have $B_1=C \ovt D$, $p_i \in A_i$, $q_1 \in C$, $q_2 \in D \ovt B_2$ and $v \, pA_1p \, v^*=qCq$ and $v \, pA_2p \, v^*=q(D \ovt B_2)q$. Hence by $(\rm iv)$, we obtain $(A_1,A_2) \sim_\infty (C, D \ovt B_2)$. Conversely, if  $(A_1,A_2) \sim_\infty (C, D \ovt B_2)$ with $B_1=C \ovt D$, then there exist type $\mathrm{I}$ factors $F_1$ and $F_2$ such that $A_1 \ovt F_1$ is unitarily conjugate to $C \ovt F_1$ in $M \ovt F_1 \ovt F_2$ which implies that $A_1 \preceq_M C$ by Lemma \ref{lemma intertwining in tensor} and therefore $A_1 \preceq_M B_1$.
\end{proof}

\begin{prop} \label{intertwine_mcduff}
Let $\mathcal{M}$ be any McDuff factor with separable predual and with two McDuff decompositions $\mathcal{M}=M_1 \ovt P_1 =M_2 \ovt P_2$. The following conditions are equivalent:

\begin{itemize}

\item [$(\rm i)$] $M_1 \preceq_{\mathcal{M}} M_2$.

\item [$(\rm ii)$]  $(M_1, P_1)$ and $(M_2,P_2)$ are stably unitarily conjugate. 

\end{itemize}
\end{prop}

\begin{proof}
$(\rm i) \Rightarrow (\rm ii)$  By Proposition \ref{prop-useful} $(\rm v)$, we know that there exists a tensor product decomposition $M_2=C \ovt D$ such that $(M_1,P_1) \sim_\infty (C,D \ovt P_2)$. Hence $P_1$ is stably isomorphic to $D \ovt P_2$ and therefore $D$ is amenable. But $M_2=C \ovt D$ is not McDuff, hence $D$ must be type $\mathrm{I}$. Therefore, by Proposition \ref{prop-useful} $(\rm ii)$, we have $(C,D \ovt P_2) \sim_\infty (C \ovt D, P_2)=(M_2,P_2)$. Finally, by Proposition \ref{prop-useful} $(\rm i)$, we conclude that $(M_1,P_1) \sim_\infty (M_2,P_2)$.

$(\rm ii) \Rightarrow (\rm i)$ follows from Proposition \ref{prop-useful} $(\rm v)$.
\end{proof}

\begin{proof}[Proof of Theorem \ref{main-thm-mcduff}]
$(\rm i) \Rightarrow (\rm ii)$ Assume that $\mathcal{M}$ has a McDuff decomposition $\mathcal{M}=M\ovt P$ where $M$ is a full factor and $P$ is a non-type ${\rm I}$ amenable factor. Fix a faithful normal conditional expectation $\rE_P : \mathcal{M} \to P$. Let $\mathcal{M}= M_0 \ovt P_0$ be another McDuff decomposition. We want to show that these two decompositions are stably unitarily conjugate. By Proposition \ref{intertwine_mcduff}, it suffices to show that $M \preceq_{\mathcal{M}} M_0$. 

Assume first that $P_0$ is of type $\mathrm{III}_1$. Then $P_0 \cong R_\infty$ is the unique Araki--Woods factor of type ${\rm III_1}$ \cite{Co75b, Co85, Ha85} and hence we can write $P_0$ as an infinite tensor product
\begin{equation}\label{eq-Araki--Woods}
(P_0, \varphi)=\bigovt_{n \in \N} (\mathbf M_2(\C) \otimes \mathbf M_2(\C),\omega_{\lambda_1} \otimes \omega_{\lambda_2})
\end{equation}
where $0 < \lambda_1, \lambda_2 < 1$, $\frac{\log \lambda_1}{\log \lambda_2} \notin \Q$ and $\omega_{\lambda_i} = \frac{1}{1 + \lambda_i}\Tr_{\mathbf M_2(\C)}(\, \cdot \, \diag(1, \lambda_i))$ for every $i \in \{1, 2\}$. Then $\varphi \in (P_0)_\ast$ is an almost periodic faithful state satisfying $((P_0)_\varphi)' \cap P_0 = \C 1$. Indeed, \eqref{eq-Araki--Woods} implies that $(P_0, \varphi) \cong (R_{\lambda_1}, \varphi_{\lambda_1}) \ovt (R_{\lambda_2}, \varphi_{\lambda_2})$ where $R_{\lambda_i}$ is the Powers factor of type ${\rm III}_{\lambda_i}$ for every $i \in \{1, 2\}$. Since $((R_{\lambda_i})_{\varphi_{\lambda_i}})' \cap R_{\lambda_i} = \C 1$ for every $i \in \{1, 2\}$ and since $(R_{\lambda_1})_{\varphi_{\lambda_1}} \ovt (R_{\lambda_2})_{\varphi_{\lambda_2}} \subset (P_0)_\varphi$, we indeed have $((P_0)_\varphi)' \cap P_0 = \C 1$ (see also \cite[Example 1.6]{Ha85}). For every $n \in \N$, define the finite dimensional unital $\ast$-subalgebra $Q_n \subset R_\infty$ by $Q_n = \bigovt_{k = 0}^n(\mathbf M_2(\C) \otimes \mathbf M_2(\C),\omega_{\lambda_1} \otimes \omega_{\lambda_2})$. We show that there exists $n \in \N$ such that $(Q_n' \cap P_0)_\varphi \preceq_{\mathcal{M}} P$.

Assume by contradiction that $(Q_n' \cap P_0)_\varphi \npreceq_{\mathcal{M}} P$ for every $n \in \N$. By \cite[Theorem 4.3 (5)]{HI15}, we can find a sequence $u_n \in \mathcal{U}((Q_n ' \cap P_0)_\varphi)$ such that $\rE_P(u_n) \to 0$ strongly as $n \to \infty$. Fix a nonprincipal ultrafilter $\omega \in \beta(\N) \setminus \N$. By construction, we have $(u_n)^\omega \in \mathcal{M}_\omega$. Hence we must have $(u_n)^\omega = (\rE_P(u_n))^\omega$ by Theorem \ref{ultraproduct_full}. This contradicts the fact that $(\rE_P(u_n))^\omega = 0$. From this contradiction, we deduce that there exists $n \in \N$ such that $(Q_n' \cap P_0)_\varphi \preceq_{\mathcal{M}} P$. Since $(Q_n' \cap P_0, \varphi|_{Q_n' \cap P_0}) \cong (P_0, \varphi)$, we have $\left( (Q_n' \cap P_0)_\varphi \right)' \cap \mathcal M = M_0 \ovt Q_n$. Then \cite[Lemma 4.9]{HI15} implies that $M \preceq_{\mathcal{M}} M_0 \ovt Q_n$. Since $Q_n$ is a type ${\rm I}$ factor, \cite[Remark 4.2 (2) and Remark 4.5]{HI15} imply that $M \preceq_{\mathcal{M}} M_0$. Assume now that $P_0$ is any non-type ${\rm I}$ amenable factor. Then we can consider $\mathcal{M} \ovt R_\infty = M \ovt (P \ovt R_\infty) = M_0 \ovt (P_0 \ovt R_\infty)$. Since $P_0 \ovt R_\infty$ is of type $\mathrm{III}_1$ by \cite[Corollary 6.8]{CT76} and amenable, we can apply the result obtained in the previous paragraph and we obtain $M \otimes \C1_{R_\infty} \preceq_{\mathcal{M} \ovt R_\infty} M_0 \otimes \C 1_{R_\infty}$. Hence, by Lemma \ref{lemma intertwining in tensor}, we finally obtain $M \preceq_{\mathcal{M}} M_0$.

$(\rm ii) \Rightarrow (\rm i)$ Assume that $\cM$ is a McDuff factor with separable predual and with a  McDuff decomposition $\cM = M \ovt P$ where $M$ is not McDuff and $P$ is any non-type ${\rm I}$ amenable factor. Assuming that $M$ is not full and following the proof of \cite[Theorem B]{Ho15}, we show the existence of $\Psi \in \Aut(\cM)$ such that $\Psi(P) \not\preceq_\cM P$. By Proposition \ref{intertwine_mcduff} and \cite[Lemma 4.9]{HI15}, this will prove that $\mathcal M$ has two McDuff decompositions that are not stably unitarily conjugate.
 
Since $M$ is not full, \cite[Theorem 3.1]{HU15} shows there exist a diffuse abelian subalgebra $B \subset M$ with faithful normal conditional expectation $\rE_B : M \to B$ and a $M$-central sequence of mutually commuting projections $(q_n)_{n\in \N}$ in $B$ such that $q_n \to \frac12$ $\sigma$-weakly as $n \to \infty$. Choose a faithful state $\phi \in M_\ast$ such that $\phi \circ \rE_B = \phi$. Since $P$ is a non-type ${\rm I}$ amenable factor, $P$ is McDuff (by combining results in \cite{Co72, Co75b, CT76, Co85, Ha85}) and we may write $P = P_0 \ovt R$ where $R$ is the hyperfinite type $\II_1$ factor and $P_0 \cong P$. Observe that $\mathcal M = M \ovt P_0 \ovt R$. Choose any faithful state $\phi_0 \in (P_0)_\ast$ and put $\vphi = \phi \otimes \phi_0 \otimes \tau \in \cM_\ast$. 

Write $(R, \tau) = \bigovt_\N (\mathbf M_2(\C), \Tr_{\mathbf M_2(\C)})$. Let $\pi_n : \mathbf M_2(\C) \to R$ be the trace preserving embedding at the $n$-th position and put 
\begin{align*}
 p_n &= \pi_n\left(\begin{pmatrix}1&0\\0&0\end{pmatrix}\right),& v_n &= \pi_n\left(\begin{pmatrix}0&1\\1&0\end{pmatrix}\right),& w_n &= 1 - 2p_n q_n. 
\end{align*}
By construction, $(w_n)_{n \in \N}$ is a sequence of mutually commuting unitaries in $\mathcal U(\cM_\vphi)$. Moreover, since $w_n^* = w_n$ for every $n \in \N$ and since $(w_n)_{n \in \N}$ is a centralizing sequence, we have that $\Ad(w_n) \to \id_{\mathcal M}$ in $\Aut(\mathcal M)$ as $n \to \infty$. Since $\mathcal{M}$ has separable predual, $\Aut(\mathcal M)$ is a Polish group. Fix a complete metric $\rm d$ on $\Aut(\mathcal M)$ compatible with the $u$-topology. We may choose inductively a subsequence $(w_{n_k})_{k\in \N}$ such that $w_{n_0} = w_0$ and 
$${\rm d}(\Ad(w_{n_0}\cdots w_{n_k}), \Ad(w_{n_0}\cdots w_{n_{k + 1}})) < 2^{-(k + 1)}$$ for every $k \in \N$. Put $W_k = w_{n_0}\cdots w_{n_k} \in \mathcal U(\mathcal M_\varphi)$ for every $k \in \N$. The triangle inequality yields ${\rm d} (\Ad(W_p), \Ad(W_q)) < 2^{- p}$ for all $q \geq p$ and hence $(\Ad(W_p))_{p \in \N}$ is a Cauchy sequence in $\Aut(\mathcal M)$. Since the metric ${\rm d}$ is complete on $\Aut(\mathcal M)$, the sequence $(\Ad(W_p))_{p \in \N}$ converges to $\Psi \in \Aut(\mathcal M)$ with respect to the $u$-topology. 

Observe that we now have a new McDuff decomposition $\mathcal M = \Psi (\mathcal M) = \Psi(M) \ovt \Psi(P)$. It remains to prove that $\Psi(P) \not\preceq_\cM P$. Note that $W_{\ell} v_{n_k} W_{\ell} = (1 - 2 p_{n_k} q_{n_k})v_{n_k}(1 - 2p_{n_k}q_{n_k}) = (1 - 2q_{n_k})v_{n_k}$ for all $\ell \geq k$ and hence $\Psi(v_{n_k}) = (1-2q_{n_k})v_{n_k}$ for every $k \in \N$. In particular, for every $x \in M$, we have
\begin{align*}
\| \rE_P(x\Psi(v_{n_k}))\|_\vphi &=  \| \rE_P(x(1-2q_{n_k})v_{n_k})\|_\vphi\\
&=  \| \rE_P(x(1-2q_{n_k}))v_{n_k}\|_\vphi \quad (\text{since } v_{n_k} \in \mathcal U(P)) \\ 
&=  \| \rE_P(x(1-2q_{n_k}))\|_\vphi \quad \quad \; \, (\text{since } v_{n_k} \in \mathcal U(\mathcal M_\varphi)) \\
&= | \vphi(x(1-2q_{n_k}))|.
\end{align*}
Since $q_{n_k} \to \frac12$ $\sigma$-weakly as $k \to \infty$, we have that $\lim_k \vphi(xq_{n_k}) = \frac12\vphi(x)$ and hence 
$$\lim_{k\to\infty} \| \rE_P(x \Psi(v_{n_k}))\|_\vphi = \lim_{k\to\infty} | \vphi(x) -2 \vphi(xq_{n_k}))| = 0.$$ 
Since $(\Psi(v_{n_k}))_{k \in \N}$ is moreover a centralizing sequence in $\mathcal M$, we obtain that 
$$\lim_k \|\rE_P(x^* \Psi(v_{n_k}) y)\|_\vphi = \lim_k \|\rE_P(x^* y\Psi(v_{n_k}))\|_\vphi = 0$$ for all $x,y \in \spn (M \cdot P)$. By \cite[Theorem 4.3 (5)]{HI15}, we obtain that $\Psi(R)\not\preceq_\cM P$. Finally, since the unital inclusion $\Psi(R) \subset \Psi(P)$ is with expectation, \cite[Lemma 4.8]{HI15} shows that $\Psi(P)\not\preceq_\cM P$.
\end{proof}

\bibliographystyle{plain}

\end{document}